\documentclass[reqno]{amsart}

\usepackage{amsmath,amssymb,amsthm}
\usepackage{graphicx,tikz}
\newtheorem{theorem}{Theorem}
\newtheorem*{thm}{Theorem}
\newtheorem{proposition}{Proposition}

\newtheorem{lemma}{Lemma}

\theoremstyle{definition}

\theoremstyle{remark}

\begin{document}

\title[]{Potential Theory and the boundary\\ of combinatorial graphs}
 \thanks{}

\author[]{Stefan Steinerberger}
\address{Department of Mathematics, University of Washington, Seattle, WA 98195, USA}
\email{steinerb@uw.edu}

\begin{abstract} Let $G=(V,E)$ be a finite, connected graph. We investigate a notion of boundary $\partial G \subseteq V$ and argue that it is well behaved from the point of view of potential theory.  This is done by proving a number of discrete analogous of classical results for compact domains $\Omega \subset \mathbb{R}^d$. These include (1) an analogue of P\'olya's result that a random walk in $\Omega$ typically hits the boundary $\partial \Omega$ within $\lesssim \mbox{diam}(\Omega)^2$ units of time, (2) an analogue of the Faber-Krahn inequality, (3) an analogue of the Hardy inequality, (4) an analogue of the Alexandrov-Bakelman-Pucci estimate, (5) a stability estimate for hot spots and (6) a Theorem of Bj\"orck stating that probability measures $\mu$ that maximize $\int_{\Omega \times \Omega} \|x-y\|^{\alpha} d\mu(x) d\mu(y)$ are fully supported in the boundary. 
\end{abstract}

\maketitle

\section{Introduction}

\subsection{Introduction}
Let $G=(V,E)$ be a finite, connected graph. We investigate a notion of `boundary', a subset of vertices $\partial G \subseteq V$.  At first, the notion of a `boundary' of a graph may appear a bit unusual -- a finite graph is a combinatorial object with no obvious complement: $G^c$ does not exist. Indeed, a graph can be fully described by its adjacency matrix $A \in \left\{0,1\right\}^{n \times n}$ and the notion of `boundary' may appear ill-defined also in that setting.
\begin{center}
\begin{figure}[h!]
\begin{tikzpicture}[scale=0.8]
\node at (0,0) {\includegraphics[width=0.2\textwidth]{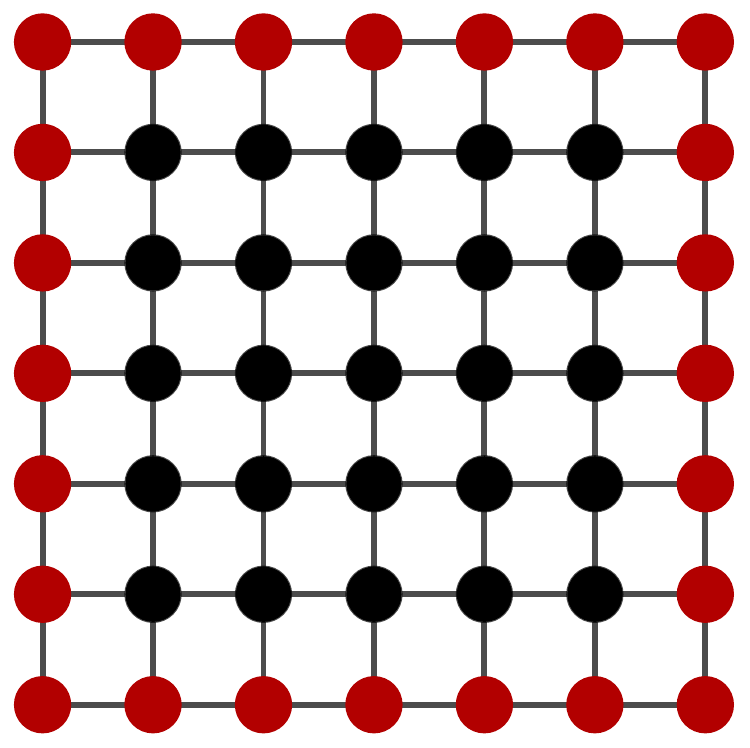}};
\node at (4.5,0) {\includegraphics[width=0.24\textwidth]{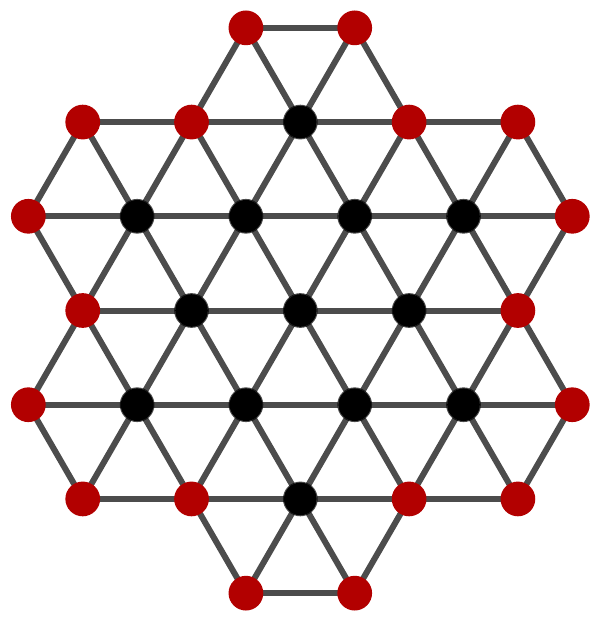}};
\node at (9.5,0) {\includegraphics[width=0.24\textwidth]{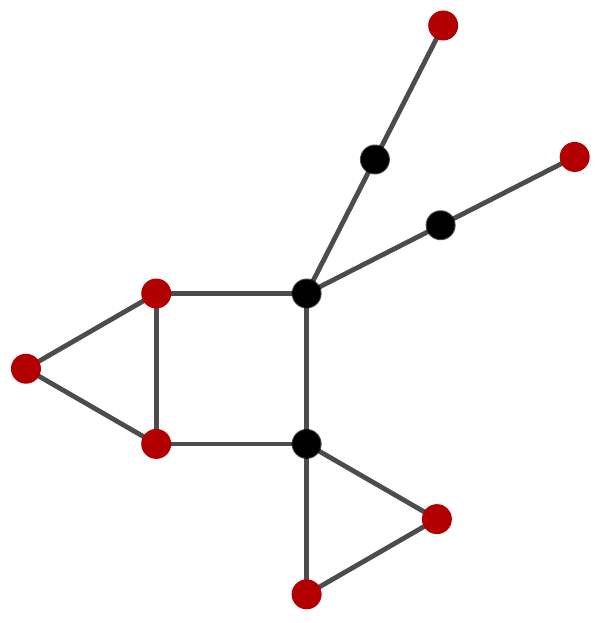}};
\end{tikzpicture}
\vspace{-10pt}
\caption{Graphs with their boundary vertices $\partial G$ in red.}
\end{figure}
\end{center}
On the other hand, graphs can also be regarded as metric spaces.  If we find ourselves in a closed room, the notion of boundary (say, a wall) makes sense independently of whether there is an external universe outside the room or not. To the best of our knowledge, such a notion of graph boundary was first proposed by Chartrand-Erwin-Johns-Zhang \cite{ch, ch2} and further developed in a variety of works, for example \cite{ce, eroh, has, mu, pel}. We will investigate a relaxation of the Chartrand-Erwin-Johns-Zhang boundary that was first proposed by the author \cite{stein}.

 \begin{quote}
 \textbf{Definition} (\cite{stein})\textbf{.} A vertex $v \in V$ is said to be in the boundary, $v \in \partial G$, if there exists another vertex $w \in V$ such that the average neighbor of $v$ is closer to $w$ than $v$ itself, formally
 $$  \frac{1}{\deg(v)} \sum_{(v,x) \in E} d(x, w) < d(v,w).$$
 \end{quote}
 
 This notion has several desirable elementary properties. If $G=(V,E)$ is a tree, the boundary $\partial G$ are exactly the vertices with degree 1 (the `leaves' of the tree).  In an arbitrary connected graph, vertices with degree 1 are always in the boundary.  Any connected graph has at least two boundary vertices with equality if and only if $G$ is a path (see also \cite{wxml}). Moreover, the definition of $\partial G$ seems to at least partially confirming to our intuition (see Figures 1,2 and 3). Moreover, if a graph `behaves' like a convex domain in $\mathbb{R}^d$, then $\partial G$ coincides with the Euclidean boundary in a suitable sense (see \cite[Proposition 4]{stein} for a precise statement).
 
 \subsection{Isoperimetric Inequality} The main result of \cite{stein} is that $\partial G$ satisfies an isoperimetric inequality. This inequality encodes the notion that a graph on many vertices should also have many vertices in its boundary unless it was somehow close to a path (since a path has only 2 boundary vertices). It is inspired by an analogous result in Euclidean space: there exists a constant $c_d > 0$ depending only on the dimension such that for all domains $\Omega \subset \mathbb{R}^d$ 
 $$ |\partial \Omega| \geq c_d \frac{ |\Omega|}{\mbox{diam}(\Omega)}.$$
 This inequality is not the classical isoperimetric inequality $|\partial \Omega| \geq c_{d,1} |\Omega|^{(d-1)/d}$, see \S 1.3, but is also sharp if and only if $\Omega$ is a ball (see Bieberbach \cite{bieberbach}).
 The main result of \cite{stein} is such an inequality for the graph boundary $\partial G$.
 \begin{thm}[Isoperimetric inequality, \cite{stein}] If $G$ is a finite, connected graph, then  
 $$ |\partial G| \geq  \frac{1}{2}\frac{1}{ \max_{v \in V} \deg(v)} \frac{|V|}{\emph{diam}(G)}.$$
 \end{thm}
 
 This inequality is sharp up to possibly the dependence on the maximal degree.  Consider the grid graph $G$ on $V = \left\{1,2,\dots, n\right\}^d$ with two vertices connected by an edge if they differ by exactly one in exactly one coordinate. Up to constants depending on $d$, we have $|\partial G| \sim n^{d-1}$ while $|V| \sim n^d$ and $\mbox{diam}(G) \sim n$.

 \vspace{-5pt}
\begin{center}
\begin{figure}[h!]
\begin{tikzpicture}[scale=0.8]
\node at (0,0) {\includegraphics[width=0.18\textwidth]{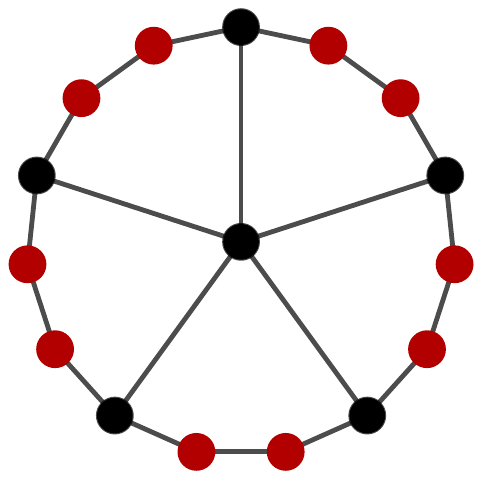}};
\node at (4.5,0) {\includegraphics[width=0.21\textwidth]{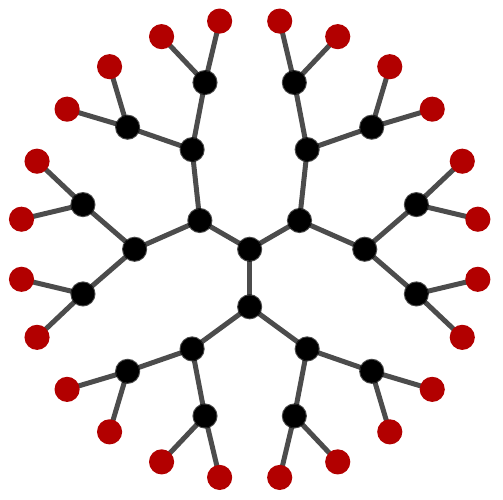}};
\node at (9.5,0) {\includegraphics[width=0.2\textwidth]{figf1.pdf}};
\end{tikzpicture}
\vspace{-10pt}
\caption{Graphs with their boundary vertices $\partial G$ in red.}
\end{figure}
\end{center}
 \vspace{-10pt}
 
 \subsection{Scaling.} The isoperimetric inequality is well-suited to illustrate an important principle when trying to prove versions of classical inequalities from Analysis. While the measure $|\Omega|$ of a domain is reasonably well captured by the number of vertices $\# V$ of a graph, there is no canonical notion of the `dimension' of a graph $G=(V,E)$. One could, conceivably, restrict one's attention to graphs that behave as if they have a certain dimension; one could, for example, require that the number of vertices in a ball scales as $\# B(v,r) \sim r^{d}$ or ask that the recurrence behavior of a random walk scales in a way that is consistent with the Euclidean case; however, the purpose of this paper is to consider universal inequalities that are true for all finite, connected graphs. \textit{If} a graph were to behave a like a well-behaved $d-$dimensional manifold, then one would expect that
 $ \mbox{diam}(\Omega) \sim |\Omega|^{1/d}.$
 Moreover, there is a classical inequality due to Bieberbach \cite{bieberbach} relating diameter and volume via $\mbox{diam}(\Omega) \geq c_{d} |\Omega|^{1/d}$ (with equality only for the ball). Thus, for arbitrary domains $\Omega \subset \mathbb{R}^d$ and arbitrary $\alpha > 0$ we have
\begin{align} \label{eq:scaling}
c_1(d, \alpha) \frac{|\Omega|}{\mbox{diam}(\Omega)^{(1-\alpha)d}} \leq  |\Omega|^{\alpha} \leq c_{2}(d, \alpha) \cdot \mbox{diam}(\Omega)^{d \alpha}
\end{align}
 The diameter of a graph $\mbox{diam}(G)$ is a completely natural quantity which suggests that \eqref{eq:scaling} might be a very natural substitute in the setting of graphs. All our inequalities become very simple when expressed in terms of powers of $\mbox{diam}(G)$.

    \vspace{-5pt}
\begin{center}
\begin{figure}[h!]
\begin{tikzpicture}[scale=0.8]
\node at (0,0) {\includegraphics[width=0.35\textwidth]{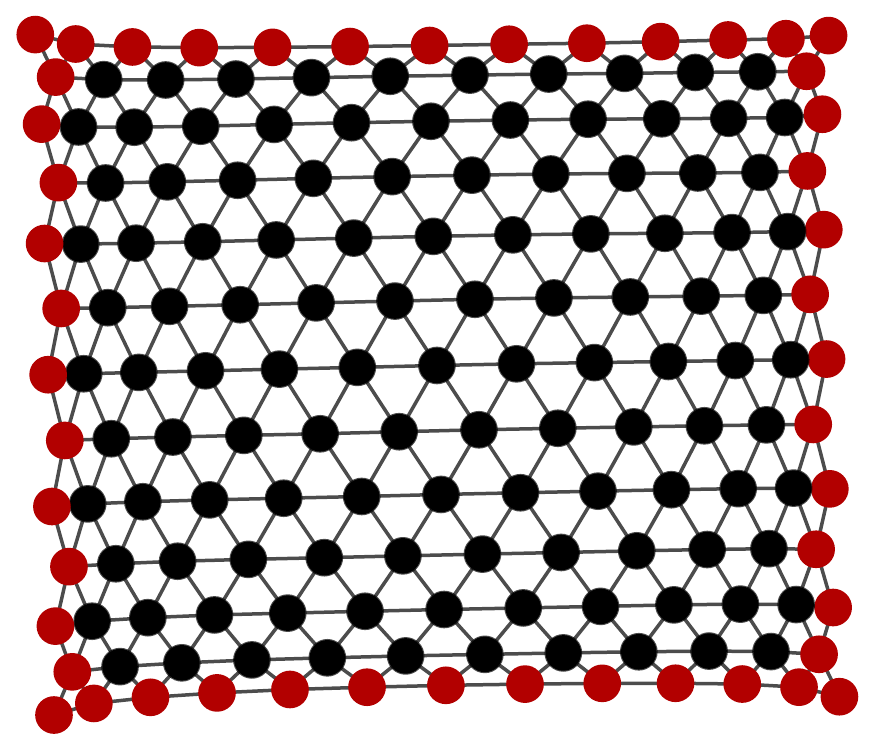}};
\node at (8,0) {\includegraphics[width=0.35\textwidth]{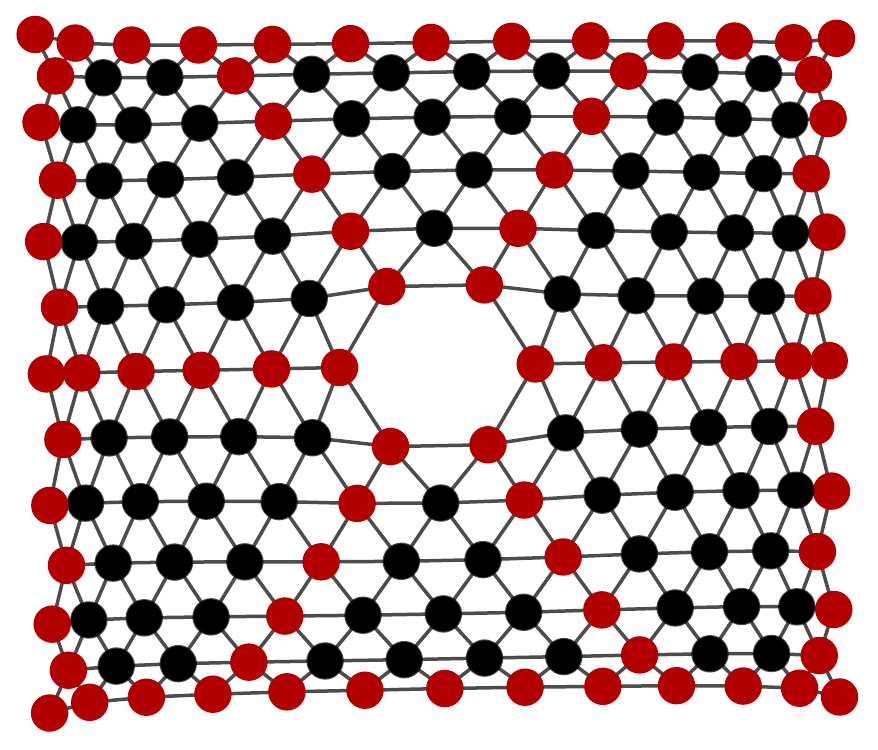}};
\end{tikzpicture}
\vspace{-10pt}
\caption{Graphs with their boundary vertices $\partial G$ in red. The removal of a single vertex can change the boundary everywhere.}
\end{figure}
\end{center}
 \vspace{-10pt}
  
 \subsection{The big picture.}
 The main purpose of this paper is to present several classical results that are true for functions $f: \Omega \rightarrow \mathbb{R}$ that vanish on the boundary of the domain $\partial \Omega$. We prove discrete versions on graphs $G=(V,E)$ for functions $f:V \rightarrow \mathbb{R}$ that vanish on the boundary $\partial G$. Graphs have traditionally been considered as objects where PDEs with (absent) Neumann boundary conditions make sense; the main purpose of this paper is to show that imposing Dirichlet boundary conditions on $\partial G$ leads to results that are consistent with the continuous theory. This suggests that $\partial G$ may be the natural notion of boundary of a graph from the point of view of potential theory and partial differential equations. This is probably not the end of the story. Just as there are \textit{many} definitions of boundary of a domain $\Omega \subset \mathbb{R}^d$ (topological boundary, measure-theoretic boundary, reduced boundary, Martin boundary...), it is quite conceivable that there are many more meaningful notions of boundary on a graph $G=(V,E)$ depending on the type of application one has in mind.

 \section{Results}

\subsection{Random walks.}
Let $\Omega \subset \mathbb{R}^d$ be a bounded domain. A random walk started in an arbitrary point inside of $\Omega$ has to leave $\Omega$ fairly quickly unless $\Omega$ is `large'. The natural continuous limit of a random walk is Brownian motion: if we pick an arbitrary point $x_0 \in \Omega$ and start Brownian motion $\omega_{x_0}(t)$ in that point, then the scaling law
$$\mathbb{E} ~\| \omega_{}(t) - \omega_{}(0)\|_{\ell^2(\mathbb{R}^d)} \sim \sqrt{t}$$
suggests that the average exit time should be when $\sqrt{t}$ is comparable to the `typical spatial extension' of $\Omega$, one way of measuring this is via the diameter $\mbox{diam}(\Omega)$.

\begin{thm}[P\'olya] Let $\Omega \subset \mathbb{R}^d$ be bounded, let $x_0 \in \Omega$ and consider Brownian motion started in $x_0$. If $T$ denote the first exit time, $\omega_{x_0}(T) \in \partial \Omega$, then
$$ \mathbb{E} ~T \leq  c_d \cdot \emph{diam}(\Omega)^2,$$
where $c_d$ depends only on the dimension.
\end{thm}

 P\'olya \cite{polya} proved the stronger bound $\mathbb{E}~T \leq c_{d,2} |\Omega|^{2/d}$ with equality and only if $\Omega$ is the ball. As discussed in \S 1.3, we rephrase it in a way that de-emphasizes dependence on the dimension. This result has a direct analogue on graphs.

\begin{theorem} Let $G=(V,E)$ be a finite, connected graph. Let $v_0 \in V$ be arbitrary and consider a random walk where $v_{k+1}$ is chosen uniformly at random from the neighbors of $v_k$. The smallest time $T$  such that $v_{T} \in \partial G$ satisfies
$$ \mathbb{E}~ T \leq  \max_{v \in V} \deg(v) \cdot \emph{diam}(G)^2.$$
\end{theorem}

This inequality is sharp up to at most the dependence on the largest degree: consider the grid graph $G=(V,E)$ on $V = \left\{1,2,\dots, n\right\}^d$ with any two vertices connected by an edge if they differ by exactly one in exactly one coordinate. In that case, the boundary $\partial G$ is exactly what one expects it to be (the `outer layer' of the cube). If one starts a random walk in the center, $x_0 = \left(\left\lfloor n/2 \right\rfloor, \left\lfloor n/2 \right\rfloor,\dots, \left\lfloor n/2 \right\rfloor\right)$, then it will take on average $\sim n^2 \sim \mbox{diam}(G)^2$ steps until one hits $\partial G$.

\subsection{The Faber-Krahn inequality} The Faber-Krahn inequality was first conjectured by Lord Rayleigh \cite{ray} when studying the theory of sound. This fundamental insight has a number of different formulations; a slightly unusual but completely elementary way of stating the result in modern language is as follows.

\begin{thm}[Faber-Krahn] Let $\Omega \subset \mathbb{R}^d$ and $f \in C^{\infty}_c(\Omega)$ be a smooth function that is compactly supported inside $\Omega$. Then
$$ \frac{\int_{\Omega} |\nabla f(x)|^2 dx}{\int_{\Omega} f(x)^2 dx} \geq \frac{c_d}{\emph{diam}(\Omega)^2},$$
where $c_d$ is a constant depending only on the dimension.
\end{thm}

As with the result in \S 2.1, our phrasing of the inequality is done so as to avoid fractional powers, Faber and Krahn showed the lower bound $\geq c_{d,2} |\Omega|^{-2/d}$. A Faber-Krahn inequality for functions $f:V \rightarrow \mathbb{R}$ on a general graph $G=(V,E)$ by itself is meaningless: while the analytic expression
$$ \frac{\int_{\Omega} |\nabla f(x)|^2 dx}{\int_{\Omega} f(x)^2 dx} \qquad \mbox{has the direct analogue} \qquad \frac{\sum_{(u,v) \in E} (f(u) - f(v))^2}{ \sum_{v \in V} f(v)^2},$$
there is no nontrivial lower bound because one can always take $f \equiv 1$ to be a constant function.  In the language of partial differential equations, the first eigenvalue of the Neumann-Laplacian is always 0 while the Faber-Krahn inequality corresponds to the smallest eigenvalue of the Dirichlet-Laplacian which is positive. We show that with this particular notion of boundary $\partial G$, forcing the function to vanish on $\partial G$ leads to an inequality with the Euclidean scaling.

\begin{theorem}[Faber-Krahn] Let $G=(V,E)$ be a finite, connected graph.  Then, for all functions $f:V \rightarrow \mathbb{R}$ that vanish on $\partial G$, we have
$$ \frac{\sum_{(u,v) \in E} (f(u) - f(v))^2}{ \sum_{v \in V} f(v)^2} \geq \frac{1}{4}\frac{\min_{v \in V} \deg(v)}{\emph{diam}(G)^2}.$$
\end{theorem}
This inequality is optimal up to possibly the dependence on $\min_{v \in V} \deg(v)$. A simple example is the path graph on $n$ vertices; the two endpoints are the boundary vertices and minimizing over all functions gives a scaling of $\mbox{diam}(G)^{-2}$. More generally, the grid graph from \S 2.1 would provide another example.

\subsection{Hardy Inequality.}
The Hardy inequality is one of the most well-known inequalities in Analysis, there are countless variants and variations; we refer to the books \cite{hardy1, hardy2}. One particular representative formulation \cite[Theorem 3.3.4]{hardy1} is as follows: if $\Omega \subset \mathbb{R}^d$ is convex and $\delta(x) = \min_{y \in \partial \Omega} \|x-y\|$ denotes the distance to the boundary, then for any $f: \Omega \rightarrow \mathbb{R}^2$ vanishing on the boundary
$$ \int_{\Omega} |\nabla f|^2 dx \geq  \frac{1}{4}\int_{\Omega} \frac{f(x)^2}{\delta(x)^2} dx.$$
There is a well-established theory of Hardy inequalities on infinite graphs with seminal work
by Keller-Pinchover-Pogorzelski \cite{keller, keller2, keller4}. There has also been much interest in Hardy inequalities for specific infinite graphs, with some emphasis on infinite lattice-type graphs, for example \cite{luz, fischer, gupta, gupta2, keller3}. Using the notion of boundary $\partial G$, we describe a general Hardy inequality on arbitrary finite graphs. We replace the distance to the boundary with $\phi:V \rightarrow \mathbb{R}$ where $\phi(v)$ denotes the expected number of steps a random walk started in $v$ needs until hitting $\partial G$.

\begin{theorem}[Hardy Inequality] Let $G=(V,E)$ be a finite, connected graph and let $\phi(v)$ be the expected number of steps a random walk started in $v \in V$ needs until hitting $\partial G$. If $f:V \rightarrow \mathbb{R}$ vanishes on $\partial G$, then
$$ \sum_{(u,v) \in E} (f(u) - f(v))^2 \geq   \sum_{v \in V} \deg(v) \frac{f(v)^2}{\phi(v)}.$$
\end{theorem}
It mirrors the classical Hardy inequality insofar as $\phi(v)$ is small if one is close to the boundary and large if one is far from the boundary.
 Theorem 3 can be considered a discrete analogue of a torsional Hardy inequality \cite{brasco}. Note that Theorem 1 implies the uniform bound $\phi(v) \leq \max_{v \in V} \deg(v) \mbox{diam}(G)^2$ which shows that Theorem 3 implies another Faber-Krahn inequality.
 Our proof follows very classical ideas and we consider it a nice example of how existing ideas can be used in conjunction with the definition of $\partial G$. More concretely, we use the Agmon-Allegretto-Piepenbrink approach \cite{agmon, all1, all2, piep} to construct a positive super-solution for an associated quadratic form: this has traditionally failed because finite, connected graphs do not support nontrivial harmonic functions. We circumvent the problem by not asking for the property to hold in $\partial G$ and compensate with $f\big|_{\partial G}=0$.

\subsection{Alexandrov-Bakelman-Pucci.}
If $\Omega \subset \mathbb{R}^d$ is bounded and $u:\Omega \rightarrow \mathbb{R}$ is a harmonic function in $\Omega$, then the maximum is attained on the boundary
$$ \forall~x \in \Omega \qquad \quad  u(x) \leq  \max_{y \in \partial \Omega} u(y).$$
There is a natural stability version: if the Laplacian $\Delta u$ is not quite zero but small, then the function should be `almost' harmonic and the statement should be `almost' true. The Alexandrov-Bakelman-Pucci estimate \cite{ale1, ale2, bakel, puc1, puc2, stein0} is one way of making this intuition precise.
\begin{thm}[Alexandrov-Bakelman-Pucci] For any bounded $\Omega \subset \mathbb{R}^d$ and $u: \Omega \rightarrow \mathbb{R}$
$$ \max_{x \in \Omega} u(x) \leq  \max_{y \in \partial \Omega} u(y) + c_d \cdot \emph{diam}(\Omega)^2 \cdot \| \Delta u\|_{L^{\infty}(\Omega)},$$
where $c_d$ depends only on the dimension.
\end{thm} 
The original ABP estimate has the quantity $c_d \cdot \mbox{diam}(\Omega) \cdot \| \Delta u\|_{L^{d}(\Omega)}$ which, following the approach outlined in \S 1.3, we rewrite using
$$\| \Delta u\|_{L^{d}(\Omega)} \leq \| \Delta u\|_{L^{\infty}(\Omega)} |\Omega|^{1/d} \leq c_d  \cdot \mbox{diam}(\Omega)\cdot  \| \Delta u\|_{L^{\infty}(\Omega)}$$
 into a quantity only containing the diameter $\mbox{diam}(\Omega)$.
One could wonder about analogous results on graphs. This requires defining the notion of harmonic function on a graph and a corresponding Laplacian operator that maps a function $f:V \rightarrow \mathbb{R}$ to another function $(Lf):V \rightarrow \mathbb{R}$.
We use a standard definition (also known as the Kirchhoff-Laplacian) that takes the mean-value theorem as its inspiration. A function $f:V \rightarrow \mathbb{R}$ is harmonic in a vertex $v \in V$ if
$$ \sum_{(v,w) \in E} (f(v) - f(w)) = 0 \qquad \mbox{or, equivalently,} \qquad \frac{1}{\deg(v)} \sum_{(v,w) \in E} f(w) = f(v).$$
The Graph Laplacian is then defined as  
$$ (Lf)(v) = \sum_{(v,w) \in E} (f(v) - f(w)).$$
  If $G=(V,E)$ is a finite, connected graph and $X \subset V$ is an arbitrary set of the vertices and if $f:V \rightarrow \mathbb{R}$ is harmonic on $V \setminus X$, then the maximum principle applies
$$ \max_{v \in V \setminus X} f(v) \leq  \max_{w \in  X} f(w).$$
  In the spirit of an Alexandrov-Bakelman-Pucci estimate, one could now wonder about the best constant $C(X)$ in the inequality
$$ \max_{v \in V} f(v) \leq  \max_{w \in  X} f(w) + C(X) \| Lf \|_{L^{\infty}(V \setminus X)}.$$
The existence of such a constant $C(X) < \infty$ is not hard to see (see \S 7.1).  
We show that if choose $X = \partial G$, then the corresponding Alexandrov-Bakelman-Pucci estimate has a constant that mirrors the Euclidean setting.
\begin{theorem}[Alexandrov-Bakelman-Pucci]  Let $G=(V,E)$ be a finite, connected graph. For any $f:V \rightarrow \mathbb{R}$
$$  \max_{v \in V} f(v) \leq  \max_{w \in \partial G} f(w) +    \frac{2 \max_{v \in V} \deg(v)}{\min_{v \in V} \deg(v)} \cdot \emph{diam}(G)^2 \cdot  \| Lf \|_{L^{\infty}(V \setminus \partial G)}.$$
\end{theorem}
The estimate is sharp up to possibly the dependence on the maximum and minimum degree. A simple example is the path graph on the vertices $\left\{0, 1, 2,  \dots, n-1, n \right\}$ with $\partial G = \left\{0, n \right\}$. Consider the function $f:V \rightarrow \mathbb{R}$ giving the expected number of steps an unbiased random walk needs to hit either $0$ or $n$. It is easy to see that $f$ vanishes in $\partial G$ and satisfies $Lf = 1$ inside the path. It is also known that $f(\left\lfloor n/2 \right\rfloor) \sim n^2$ which is exactly the rate predicted by Theorem 4. The proof of Theorem 4 shows a slightly stronger result: instead of considering $Lf$, it suffices to consider $ \max \left\{(Lf)(v),0 \right\}$, the only difficulty comes from vertices where $(Lf)(v) > 0$, these being the ones exceeding the average value of their neighbors.

\subsection{Hot Spots Inequalities.}
One of the more mysterious problems in Partial Differential Equations is the \textit{Hot Spots Conjecture} of Jeffrey Rauch.  The observation is that if $\Omega \subset \mathbb{R}^d$ is a bounded domain, then the first nontrivial eigenfunction of the Neumann Laplacian, $-\Delta \phi = \lambda_2 \phi$, describes the long-term behavior of the heat equation. It is also the solution of the variational problem
$$ \phi = \arg\min_{\int_{\Omega} f(x) dx = 0} \frac{ \int_{\Omega} |\nabla f(x)|^2 dx}{ \int_{\Omega} f(x)^2 dx} .$$
 Rauch asked whether $\phi$ attains its maximum and minimum value on the boundary of the domain.
What makes the entire question somewhat mysterious is that
\begin{itemize}
\item it is not always the case, maxima and minima can be inside the domain,
\item but counterexamples are \textit{very} difficult to construct \cite{burdzy} and
\item it is not clear what the right assumptions on $\Omega$ are supposed to be.
\end{itemize}
It was widely assumed that $\Omega \subset \mathbb{R}^d$ being convex should be sufficient for the statement to be true, however, this was recently disproven by Jaume de Dios Pont \cite{pont} in sufficiently high dimensions. For convex $\Omega \subset \mathbb{R}^2$, the statement may still be true. Another widely circulated conjecture is that for planar domains $\Omega \subset \mathbb{R}^2$ it may be enough to require that $\Omega$ is simply connected. The very same question can naturally be asked on graphs \cite{dep, ken1, ken2, led}: if $G=(V,E)$ is a connected graph and $f:V \rightarrow \mathbb{R}$ is an eigenvector of the Graph Laplacian corresponding to the smallest positive eigenvalue, where can $f$ assume its maximum and minimum?

\vspace{5pt}
\begin{center}
\begin{tikzpicture} 
 \node at (0,0) {\includegraphics[width=0.45\textwidth]{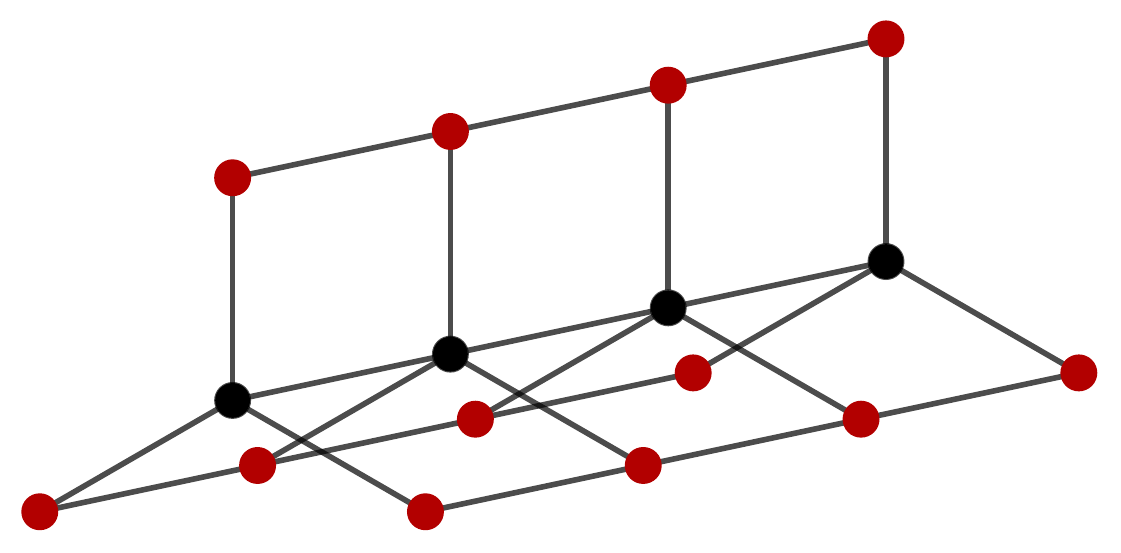}};
 \node at (6,0) {\includegraphics[width=0.45\textwidth]{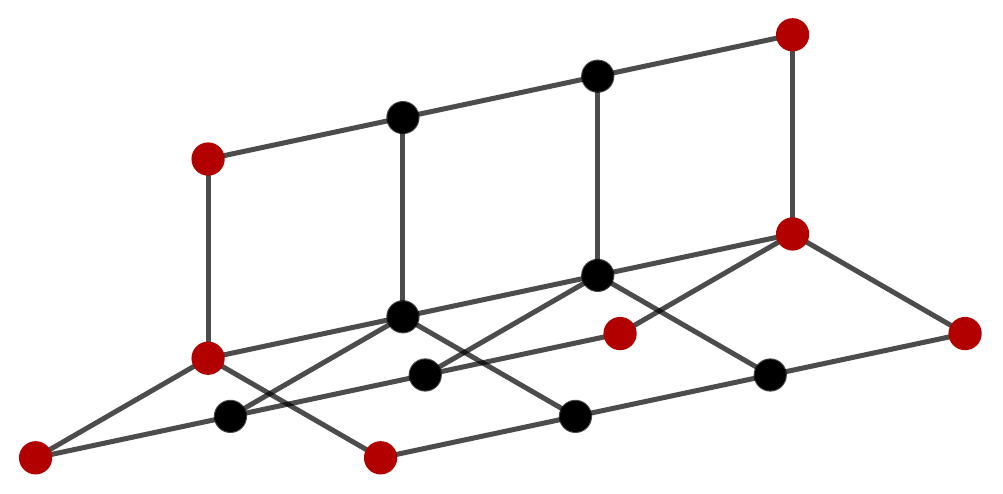}};
 \node at (0, -1.7) {The boundary $\partial G$};
  \node at (6, -1.7) {Location of hot/cold spots};
\end{tikzpicture}
\end{center}

More precisely, if one claims to have some notion of boundary, is it true that the largest/smallest entry of the eigenvector tend to be assumed in $\partial G$? We observe that this seems to be the case for most (but not all) graphs.

\begin{quote}
\textbf{Observation.} For \textit{most} graphs $G=(V,E)$, the second eigenvector seems to attain its maximum and minimum in $\partial G$.
\end{quote}
Much like in the Euclidean case, counterexamples seem to be rare and difficult to find: the ones shown are taken from a handful examples that were found after checking the $>10000$ examples stored in Mathematica's database.

\vspace{5pt}
\begin{center}
\begin{tikzpicture} 
 \node at (0,0) {\includegraphics[width=0.25\textwidth]{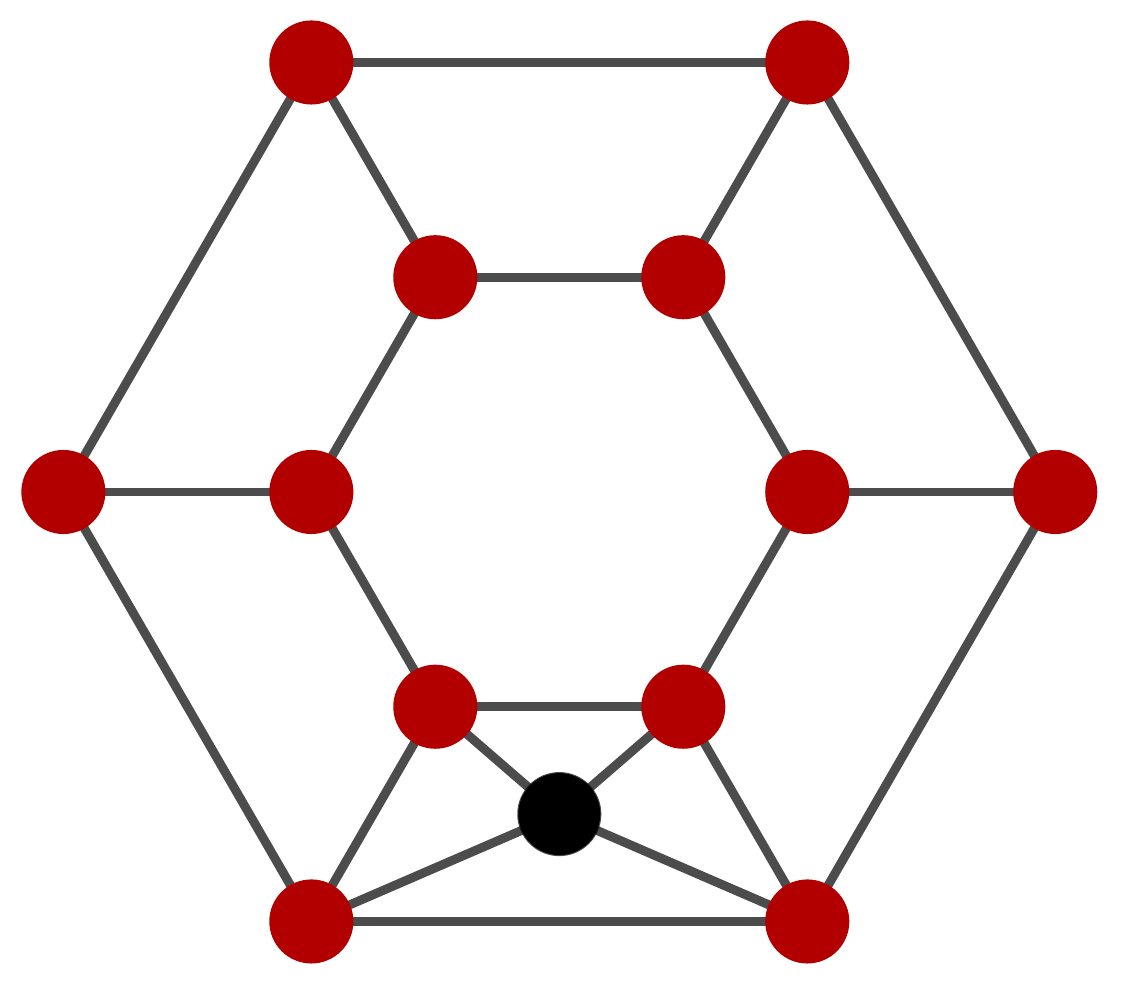}};
 \node at (6,0) {\includegraphics[width=0.25\textwidth]{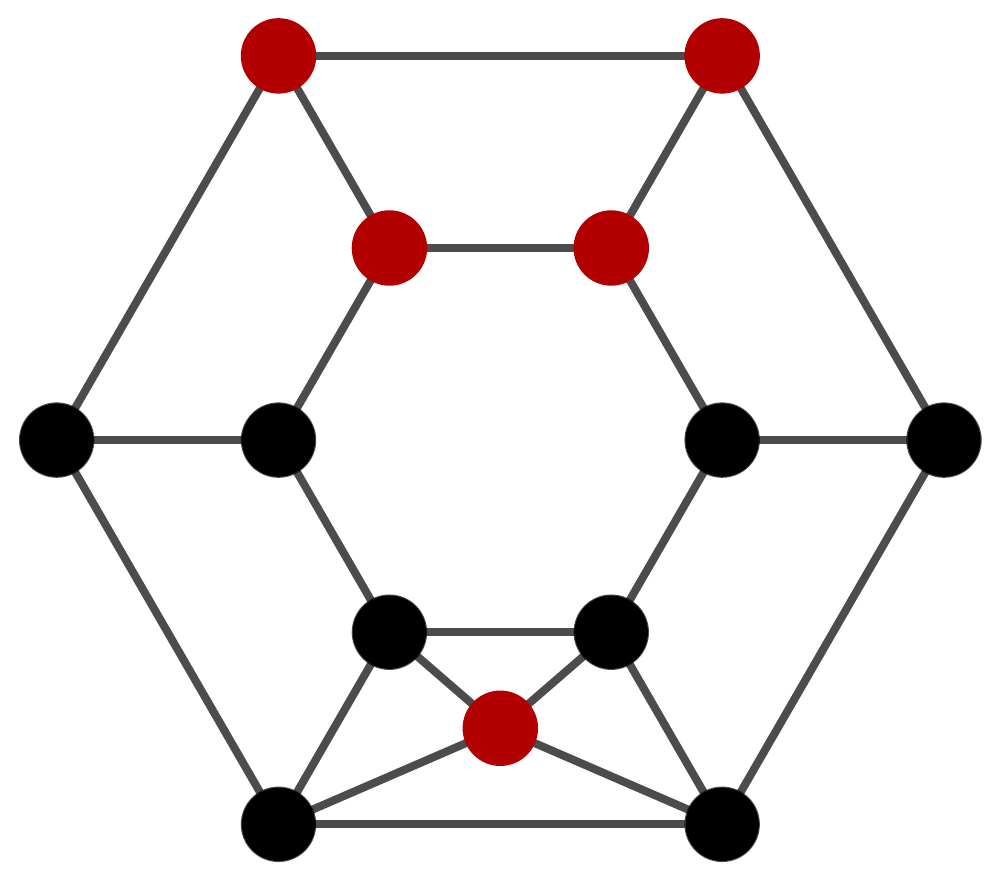}};
 \node at (0, -1.9) {The boundary $\partial G$};
  \node at (6, -1.9) {Location of hot/cold spots};
\end{tikzpicture}
\end{center}
\vspace{5pt}

 We believe this to be, perhaps in a subtle way, another indicator that $\partial G$ is the `right' definition: it seems to preserve the mysterious character of the Hot Spots conjecture in the sense that it is false but only rarely so. Regarding the (continuous) Hot Spots conjecture, some positive results are known.

\begin{thm}[\cite{stein4}] If $\Omega \subset \mathbb{R}^d$ is a bounded domain with smooth boundary and $u$ is an eigenfunction of the Laplacian with Neumann boundary conditions corresponding to the smallest positive eigenvalue, then
$$ \max_{x \in \Omega} u(x) \leq 60 \max_{x \in \partial \Omega}u(x).$$
\end{thm}

The constant 60 has since been improved by Mariano-Panzo-Wang \cite{mariano}.  We will show that a similar result is possible on general graphs.
 
\begin{theorem} Let $G=(V,E)$ be a finite and connected graph, let $Lf = \lambda_2 f$ where $\lambda_2$ is the smallest positive eigenvalue and $\lambda_2 < 1$. Then
  $$ \max_{v \in V \setminus \partial G} f(v) \leq   \left(1 - \frac{\lambda_2}{\min_{v \in V} \deg(v)}\right)^{- 2\max_{v \in V} \emph{\tiny diam}(G)^2}  \max_{w \in \partial G}f(w).$$
\end{theorem}
This result is not as universal as the continuous analogue. It is uniformly bounded for a large family of path-like graphs for which $\lambda_2 \lesssim 1/\mbox{diam}(G)^2$. It would be desirable to have better upper bounds and a better understanding of how the constant in Theorem 5 has to depend on the graph. It is an interesting question whether the discrete Hot Spots problem with $\partial G$ as the notion of boundary can be useful in obtaining some robust arguments that can be used in the Euclidean setting; this seems potentially very difficult, however, at least numerical experimentation is much simpler (linear algebra as opposed to high-precision numerical solutions of a partial differential equation).

\subsection{Bj\"orck's Theorem.} 
We conclude with a classic 1956 result of Bj\"orck \cite{bjorck}.  If $\Omega \subset \mathbb{R}^d$ is a compact set, what can be said about the probability measure $\mu$ supported on $\Omega$ for which
$$ \int_{\Omega} \int_{\Omega} \|x-y\|^{\alpha} d \mu(x) d \mu(y) \rightarrow \max.$$

\begin{thm}[Bj\"orck, 1956] For any $\alpha \geq 1$, the support of the extremal measure $\mu$ is contained in the boundary, $ \emph{supp}(\mu) \subseteq \partial \Omega$.
\end{thm}

Bj\"orck shows that if $\Omega \subset \mathbb{R}^d$ with $d \geq 2$, then $\alpha > 0$ is sufficient. $\alpha \geq 1$ is needed when $d=1$. Bj\"orck's result is naturally related to certain embedding problems and has inspired subsequent work in distance geometry \cite{alex, alex3, carando, gross, hinrichs, schon1, schon2, wolf2}.
We prove an analogue of Bj\"orck's result for general graphs.

\begin{theorem} If $f:\mathbb{R}_{\geq 0} \rightarrow \mathbb{R}$ is a strictly convex and monotonically increasing function, then any extremal probability measure $\mu$ on the vertices maximizing
$$ \sum_{v,w \in V} f(d(v,w)) \mu(v) \mu(w) \rightarrow \max$$
is supported in the boundary, $\emph{supp} (\mu) \subseteq \partial G$.
\end{theorem}

This covers the case $d(x,y)^{\alpha}$ for $\alpha > 1$. The case $\alpha = 1$, corresponding to the energy $\sum_{v,w} d(v,w) \mu(v) \mu(w)$ is more subtle: our argument shows that for any extremizer there exists an extremizer fully supported in the boundary $\partial G$ attaining the same value. This case is known to give rise to fairly sophisticated geometry in the graph case, we refer to \cite{stein2, stein3, thomassen}.

\section{Some preparatory statements}
We start with two easy statements.  They will be used, implicitly or explicitly, in all subsequent arguments. We first recall the definition of a boundary vertex: a vertex $v \in V$ is said to be in the boundary, $v \in \partial G$, if there exists another vertex $w \in V$ such that a randomly chosen neighbor of $v$ is closer to $w$ than $v$ itself. It may help to think of $v$ as a \textbf{v}ertex and of $w$ as a \textbf{w}itness (that witnesses the property of $v \in \partial G$).  Another way of thinking about the definition of $\partial G$ is as follows:  we think of $w$ as a fixed vertex and partition all the remaining vertices of $V$ by their distance from $w$, meaning
$$ V = \left\{w\right\} \cup \bigcup_{i=1}^{\mbox{\tiny diam}(G)}A_i \qquad \mbox{where} \qquad A_i = \left\{v \in V: d(v,w) = i \right\}.$$ 
One advantage of this ordering is that vertices in $A_i$ can only be connected to vertices in $A_{i-1}$ and $A_{i+1}$, we induce a type of foliation. Let us now suppose that $v \in A_i$.  The vertex $v$ has to have at least one neighbor in $A_{i-1}$ because of the graph $G$ is connected and the shortest path from $w$ to $v$ has to pass through $A_{i-1}$. The vertex $v$ may or may not have other neighbors in $A_{i-1}, A_i$ and $A_{i+1}$.
\begin{proposition}
The vertex $v \in V$ is witnessed by $w \in V$ to be in the boundary if and only if $v$ has more neighbors in $A_{i-1}$ than in $A_{i+1}$.
\end{proposition}
\begin{proof} Abbreviate $d(v,w) = i$ and write $N(v)$ for the set of vertices that are adjacent to $v$.  Then all the neighbors of $v$, meaning vertices $x \in V$ such that $(v,x) \in E$, are distance $i-1, i$ or $i+1$ from $w$.  
By assumption, $v$ is witnessed by $w$ if 
\begin{align*}
  i =  d(v,w) &> \frac{1}{\deg(v)} \sum_{(v,x) \in E} d(x, w) \\
  &= \frac{i-1}{\deg(v)} \#(A_{i-1} \cap N(v))+  \frac{i}{\deg(v)} \# (A_i \cap N(v)) \\
  &+ \frac{i+1}{\deg(v)} \# ( A_{i+1} \cap N(v)).
  \end{align*}
The cardinality of the three sets sums up to $\deg(v)$, this can be rearranged as
$$ 0 > - \#(A_{i-1} \cap N(v)) +  \# ( A_{i+1} \cap N(v))$$
which is the desired statement.
\end{proof}

\begin{center}
\begin{figure}[h!]
\begin{tikzpicture}[scale=2]
\filldraw (-0.5,0) circle (0.03cm);
\draw [thick, dashed] (-0.5, 0) -- (-0.2, 0.4);
\draw [thick, dashed] (-0.5, 0) -- (-0.2, 0.2);
\draw [thick, dashed] (-0.5, 0) -- (-0.2, 0.0);
\draw [thick, dashed] (-0.5, 0) -- (-0.2, -0.3);
\node at (-0.6, -0.1) {$w$};
\draw[ultra thick] (1,0) ellipse (0.3cm and 0.6cm);
\draw[ultra thick] (2,0) ellipse (0.3cm and 0.6cm);
\draw[ultra thick] (3,0) ellipse (0.3cm and 0.6cm);
\node at (1, -0.85) {$A_{i-1}$};
\node at (2, -0.85) {$A_i$};
\node at (3, -0.85) {$A_{i+1}$};
\filldraw (2, 0.4) circle (0.03cm);
\filldraw (2, -0.15) circle (0.03cm);
\filldraw (2, -0.4) circle (0.03cm);
\draw [thick] (1, 0.5) -- (2, 0.4) -- (3, 0.5);
\draw [thick] (1, 0.3) -- (2, 0.4) -- (3, 0.3);
\draw [thick] (1, 0.2) -- (2, 0.15);
\draw [thick] (1, 0.15) -- (2, 0.15);
\draw [thick] (1, 0.1) -- (2, 0.15);
\draw [thick] (1, -0.2) -- (2, -0.15);
\draw [thick] (1, -0.15) -- (2, -0.15);
\draw [thick] (1, -0.1) -- (2, -0.15) -- (3, -0.15);
\draw [thick] (1, -0.5) -- (2, -0.4) -- (3, -0.5);
\draw [thick] (2, -0.4) -- (3, -0.4);
\draw [thick] (2, 0.15) -- (2,0.4);
\filldraw[red] (2, -0.15) circle (0.03cm);
\filldraw[red] (2, 0.15) circle (0.03cm);
\end{tikzpicture}
\vspace{-10pt}
\caption{Vertices at distance $i-1, i$ and  $i+1$ from $w$. Two vertices (red) are being identified by $w$ as being boundary vertices.}
\end{figure}
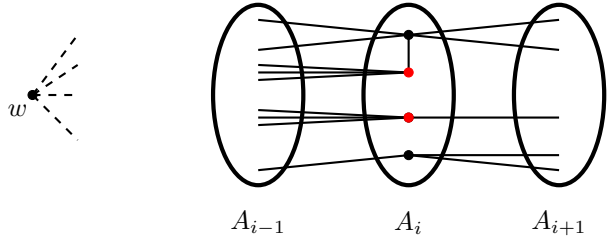
\end{center}

We note that a vertex $v \in V$ is in the boundary if it is witnessed as being in the boundary by one other vertex $w \in V$.  In particular, if a vertex $v \in V$ is \textit{not} in the boundary, then this gives us information about every other vertex $w \in V$. Negating Proposition 1, we arrive at the equally useful Proposition 2.

\begin{proposition}
A vertex $v$ is in the interior, $v \in V \setminus \partial G$, if and only if, for every other vertex $w$, the number of neighbors of $v$ that are further away from $w$ is at least as large as the number of neighbors of $v$ that are closer to $w$.
\end{proposition}

\section{Proof of Theorem 1}
The proof of Theorem 1 decouples into two parts: a combinatorial reduction that encapsulates the main idea and some routine computations in probability theory. We discuss these separately.

\subsection{The Idea.} Suppose $v \in V$ is an arbitrary vertex and we start a random walk in $x_0 = v$ where $x_{k+1}$ is a randomly chosen neighbor of $x_k \in V$. The goal is to show that we will eventually hit a vertex $v_k \in \partial G$ that is in the boundary and to quantify how long that will take and how unlikely it is that it will take a very long time.  The main idea behind the proof is to keep track of the random sequence of integers $(a_k)_{k=1}^{\infty}$ where $a_k = d(v,v_k)$. We may think of  $(a_k)_{k=1}^{\infty}$ as a random walk on the first few integers, more precisely on
$$ \mbox{the state space} \quad \left\{0,1,2,\dots, \max_{w \in V} d(v,w)\right\} \subseteq \left\{0,1,2,\dots,  \mbox{diam}(G)\right\}.$$
We note that $a_{k+1} - a_k \in \left\{-1,0,1\right\}$ and the precise transition probabilities depend on the vertex $v_k$. However, as long as $v_k \notin \partial G$, Proposition 2 implies that
\begin{equation} \label{eq:drift}
\mathbb{P}\left( a_{k+1} = a_k + 1\right) \geq \mathbb{P}\left( a_{k+1} = a_k - 1\right).
\end{equation}

Moreover, because the graph $G$ is connected, there is a shortest path from $v$ to $v_k$, so at least one of the neighbors of $v_k$ is closer to $v$. Since there is a uniform upper bound $\deg(v_k) \leq \max_{v \in V} \deg(v)$, the likelihood of the random walk decreasing its distance to the origin is always at least
$$  \mathbb{P}\left( a_{k+1} = a_k - 1\right) \geq \frac{1}{\max_{v \in V} \deg(v)}.$$

\begin{center}
\begin{figure}[h!]
\begin{tikzpicture}
\filldraw (0.5,0) circle (0.06cm);
   \draw [black,thick,domain=330:360] plot ({cos(\x)}, {sin(\x)});
   \draw [black,thick,domain=0:30] plot ({cos(\x)}, {sin(\x)});
      \draw [black,thick,domain=335:360] plot ({2*cos(\x)}, {2*sin(\x)});
   \draw [black,thick,domain=0:25] plot ({2*cos(\x)}, {2*sin(\x)});
         \draw [black,thick,domain=340:360] plot ({3*cos(\x)}, {3*sin(\x)});
   \draw [black,thick,domain=0:20] plot ({3*cos(\x)}, {3*sin(\x)});
         \draw [black,thick,domain=345:360] plot ({4*cos(\x)}, {4*sin(\x)});
   \draw [black,thick,domain=0:15] plot ({4*cos(\x)}, {4*sin(\x)});
\filldraw (3.4,0.2) circle (0.06cm);
\filldraw (3.8,-0.5) circle (0.06cm);
\filldraw (2.4,-0.4) circle (0.06cm);
\filldraw (2.7,0.4) circle (0.06cm);
\filldraw (1.3,0) circle (0.06cm);
   \node at (0.5, -0.3) {$v$};
   \draw [thick, ->]  (2.4, -0.4) -- (1.45, -0.1);
   \draw [thick, ->] (1.3, 0) -- (2.6, 0.38);
   \draw [thick, ->] (2.7, 0.4) -- (3.3, 0.2);
   \draw [thick, ->] (3.4, 0.2) -- (3.75, -0.4);
\end{tikzpicture}
\caption{A random walk on the graph where we only keep track of the distance to $v \in V$ introduces a random walk on the integers.}
\end{figure}
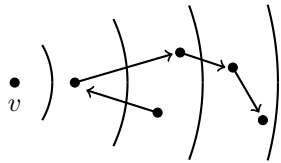
\end{center}

We can now introduce a second type of random variable, the sequence $(b_{\ell})_{\ell=1}^{\infty}$ which records all the values attained by the sequence $a_k$ in the correct order with consecutive repetitions removed. Phrased differently, the sequence $(b_{\ell})_{\ell=1}^{\infty}$ encodes all the values attained by $a_k = d(v,v_k)$ in the correct order but does not keep track of how long $a_k$ may stay the same distance. This implies, by construction, that $b_{\ell+1} = b_{\ell} \pm 1$ is a classical random walk on the integers in the sense that we move either left or right (though not necessarily with the same likelihood). As long as the underlying vertex $v_k \notin \partial G$, we have
\begin{align} \label{tokyo:drift}
 \mathbb{P}\left( b_{\ell+1} = b_{\ell} + 1\right) \geq \mathbb{P}\left( b_{\ell+1} = b_{\ell} - 1\right).
 \end{align}
 Note also that if $b_{\ell} = 0$, then $v_k = v$ and $b_{\ell+1} = 1$. The question is now whether we can bound the first time that $v_k \in \partial G$. We know that this is surely the case once $a_k =  \max_{w \in V} d(v,w)$ but it may happen sooner than that.
 The main idea is now as follows: as long as $v_k \notin \partial G$ is not part of the boundary, the random walk $a_k$ behaves like a random walk on $\left\{0,1,2,\dots, \max_{w \in V} d(v,w)\right\}$. The random walk is not necessarily unbiased but if it has a drift, then \eqref{eq:drift} and \eqref{tokyo:drift} ensure that this drift can only be to the right. Such a random walk will, at some point, attain large values and, in particular, a value larger than $\mbox{diam}(G)$. This is impossible. We can thus bound the expected time it takes for $v_k$ to hit $\partial G$ by the expected time it takes for $(a_k)_{k=1}^{\infty}$ to exceed $\mbox{diam}(G)$ (which, as long as it never hits a boundary vertex, it is eventually going to do).  This, expected hitting time we can bound in terms of the expected hitting time for $(b_k)_{k=1}^{\infty}$ up to a scaling factor that will concentrate.
 
 \subsection{Probabilistic Lemmas.} We will now compute the expected behavior of these objects.  These are all fairly classical problems in probability theory. The focus is on a simple and self-contained exposition (at the cost of slightly suboptimal constants). Suppose we
 are given a random walk on  $\left\{0,1,2,\dots, \max_{w \in V} d(v,w)\right\}$ where $0$ is repelling and the largest number is absorbing. We may assume, by stochastic domination (equation \eqref{tokyo:drift}) that the random walk is unbiased.
 By symmetry, we may think of this as an unbiased random walk on 
 $$\left\{-\max_{w \in V} d(v,w), \dots, -1, 0,1,\dots, \max_{w \in V} d(v,w)\right\}$$
  with absorbing endpoints. This, in turn, is dominated by working on the larger interval $\left\{-\mbox{diam}(G), \dots, -1, 0,1,\dots, \mbox{diam}(G)\right\}$. We can now deduce an elementary and well-known Lemma ensuring that the random walk exits quickly.  In terms of the big picture, this ensures that the random walk $b_k$ is likely to exit within a controlled amount of time.
 
\begin{lemma}
Consider a random walk $\left\{-\emph{diam}(G), \dots, -1, 0,1,\dots, \emph{diam}(G)\right\}$ starting in $x_0$. The first hitting time $T$ of $\pm \emph{diam}(G)$ satisfies
$$ \mathbb{E}~T \leq \emph{diam}(G)^2.$$
Moreover, for any $k \geq 4 \cdot \emph{diam}(G)^2$, we have
 $$ \mathbb{P}\left( T \geq k \right) \leq 4 \cdot 2^{ - k/(2\emph{diam}(G)^2)}.$$
\end{lemma}
\begin{proof} We consider the expected hitting time when starting in $x_0$ as a function $T(x_0)$. Then $T$ vanishes at the endpoints
and satisfies the equation
$$ T(x) =   \frac{T(x-1) + T(x+1)}{2} + 1$$
everywhere else. This can be solved explicitly and is given by the explicit function $ T(x) = \mbox{diam}(G)^2 - x^2.$
The maximum is attained in $T(0) = \mbox{diam}(G)^2$. This argument can now be bootstrapped. By Markov's inequality, we have
 $$ \mathbb{P}\left(T \geq 2 \cdot \mbox{diam}(G)^2\right) \leq \frac{1}{2}.$$
 Using the Markov property, we deduce, for any $\ell \in \mathbb{N}_{\geq 1}$ that
  $$ \mathbb{P}\left(T \geq 2\ell \cdot \mbox{diam}(G)^2\right) \leq \frac{1}{2^{\ell}}.$$
This implies the result.
 \end{proof}

 \subsection{Proof of Theorem 1}
 \begin{proof}
 The previous section provides good control on the random walk $(b_{\ell})_{\ell=1}^{\infty}$. However, we need control on $(a_k)_{k=1}^{\infty}$ since
 that is the random walk that actually corresponds to the random walk on the graph. The main difference is that we 
 have $b_{\ell+1} = b_{\ell} \pm 1$ where as $a_{k+1} = a_k$ is quite possible. However, $(a_k)_{k=1}^{\infty}$ is unlikely to be stationary for a very long time since
 $$ \mathbb{P}\left( a_{k+1} = a_k + 1\right) \geq \mathbb{P}\left( a_{k+1} = a_k - 1\right) \geq \frac{1}{\max_{v \in V} \deg(v)}.$$ 
 Each step of the random walk either increases or decreases its distance with likelihood at least 
$$ \frac{2}{\deg(v_k)} \geq \frac{2}{\max_{v \in V} \deg(v)}.$$
The relevant question is now: if we have to make $T$ jumps of the random walk $(b_{\ell})_{\ell=1}^{\infty}$ to ensure an escape, how many steps of $(a_k)_{k=1}^{\infty}$ are required to have at least $T$ jumps in $b_{\ell}$? If we use the random variable $Y$ to denote the number of jumps needed to either increase or decrease the distance from $v$, then the Markovian property of the random walk ensures that
$$ \mathbb{P}\left( Y \geq k \right) \leq \left(1 - \frac{2}{\max_{v \in V} \deg(v)}\right)^{k-1}.$$
This implies
$$ \mathbb{E}Y = \sum_{k=1}^{\infty} \mathbb{P}(Y \geq k) \leq \sum_{k=1}^{\infty} \left(1 - \frac{2}{\max_{v \in V} \deg(v)}\right)^{k-1} = \frac{\max_{v \in V} \deg(v)}{2}.$$
This shows that we expect to need typically no more than $\sim \max_{v \in V} \deg(v)/2$ steps to see a change in the distance. If we want to see $T$ changes in the distance, this can typically be accomplished in typically $T \max_{v \in V} \deg(v)/2$.  Wald's identity \cite{wald1, wald2} states that if we have a random sum of identical random variables
$$ Z = X_1 + \dots + X_N \qquad \mbox{with}~N~\mbox{a random variable},$$
then $\mathbb{E}Z = (\mathbb{E} X)( \mathbb{E} N)$.  In our case this shows that
$$ \mbox{E}\left[\mbox{number of steps}\right] \leq \left(\mathbb{E}T \right) \left( \mathbb{E}Y \right) \leq \max_{v \in V} \deg(v) \cdot \mbox{diam}(G)^2.$$
\end{proof}

We observe that the entire argument can be translated, almost verbatim, into the continuous setting.  If we consider Brownian motion $\omega(t)$ started in a point $x_0 \in \Omega \subset \mathbb{R}^d$, then the random variable keeping track of the distance to $x_0$
$$ X_t = \| \omega(t) - x_0\|_{\ell^2(\mathbb{R}^d)}$$
is a Bessel process of order $d$ satisfying the SDE
$$ d X_t = \frac{d-1}{2} \frac{dt}{X_t} + dW_t,$$
where $W_t$ is one-dimensional Brownian motion. The Bessel process may be considered a one-dimensional random walk with a drift to the right. This drift, geometrically interpreted, is a consequence of the curvature of the sphere and, more precisely, the fact that slightly more than half the points in a small ball centered around $\omega(t)$ are further away from $x_0$ than $\omega(t)$ is (this is \textit{exactly} our definition of interior point in the graph setting). Continuing with the argument, we see that the likelihood of $X_t$ exceeding $\mbox{diam}(\Omega)$ is increasing quickly as soon as $t$ exceeds $\mbox{diam}(\Omega)^2$. This would give a probabilistic proof of P\'olya's result (though not necessarily of the inequality $\mathbb{E}T \leq c_d |\Omega|^{2/d}$ which requires symmetrization techniques). Arguments of this type have been used before in the continuous setting to study the torsion function \cite{hoskins} as well as Hermite-Hadamard inequalities \cite{lu, stein6}.

\section{Proof of Theorem 2}
 \begin{proof}
  Let $G=(V,E)$ be a finite, connected graph. We consider a minimizer of the variational problem
$$ \frac{\sum_{(u,v) \in E} (f(u) - f(v))^2}{ \sum_{v \in V} f(v)^2} \rightarrow \min \qquad \quad \mbox{subject to}~ f\big|_{\partial G} = 0.$$
 The existence of such a minimizer is obvious from compactness. These minimizers are actually easy to compute: if we drop the condition $f\big|_{\partial G} = 0$, then 
$$\sum_{(u,v) \in E} (f(u) - f(v))^2 = \left\langle f, L f \right\rangle,$$
where $L$ is the Graph Laplacian corresponding to $G$. The expression
$$ \frac{\sum_{(u,v) \in E} (f(u) - f(v))^2}{ \sum_{v \in V} f(v)^2} = \frac{\left\langle f, Lf \right\rangle}{\left\langle f, f \right\rangle}$$
is merely the Rayleigh-Ritz quotient and a minimizer is given by an eigenvector corresponding to the smallest eigenvalue: this eigenvalue is 0 and the minimizer is given by a constant function $f \equiv 1$. The first `nontrivial' solution is given by an eigenvector corresponding to the smallest positive eigenvalue (and, by orthogonality, this function will have mean value 0).
We are considering a different problem since we require that the function $f$ vanishes on $\partial G \subset V$ but not identically. It is not too difficult to see that this also corresponds to an eigenvalue for a different matrix. Let $L_2$ be the symmetric matrix that arises from taking $L$ and erasing all the rows and columns that correspond to a vertex in $\partial G$.  Alternatively, we could define $L_2$ as the Laplacian of the subgraph defined on $V \setminus \partial G$ after additionally adding $+1$ to each diagonal entry $(L_2)_{ii}$ for every neighbor of $v_i$ that is in the boundary. This can also be seen by writing the quadratic form as
\begin{align*}
\sum_{(u,v) \in E} (f(u) - f(v))^2  &= \sum_{(u,v) \in E \atop u,v \in V \setminus \partial G} (f(u) - f(v))^2 + \sum_{(u,v) \in E \atop u \in V \setminus \partial G, v \in \partial G} (f(u) - f(v))^2 \\
&= \sum_{(u,v) \in E \atop u,v \in V \setminus \partial G} (f(u) - f(v))^2 + \sum_{(u,v) \in E \atop u \in V \setminus \partial G, v \in \partial G} f(u)^2 \\
&= \sum_{(u,v) \in E \atop u,v \in V \setminus \partial G} (f(u) - f(v))^2 + \sum_{u \in V} f(u)^2 \# \left\{(u,v) \in E: v \in \partial G \right\}.
\end{align*}

Hence, any function $f:V \rightarrow \mathbb{R}$ vanishing on $\partial G$ satisfies
$$\sum_{(u,v) \in E} (f(u) - f(v))^2 = \left\langle f, L_2 f \right\rangle$$
and the minimizer has to be an eigenvector corresponding to the smallest eigenvalue. Moreover, the trivial inequality $||f(v)| - |f(w)|| \leq |f(v)-f(w)|$ allows us to assume without loss of generality that $f$ assumes only non-negative values.
It is easy to see that the smallest eigenvalue of $L_2$ has to be positive: the function is not allowed to vanish identically and thus it changes values across at least one edge. In summary,
$$ \min_{f:V \rightarrow \mathbb{R} \atop f |_{\partial G} = 0 } \frac{\sum_{(u,v) \in E} (f(u) - f(v))^2}{ \sum_{v \in V} f(v)^2} = \lambda_1(L_2) > 0.$$

We shall now fix $f$ to be a solution of $L_2 f = \lambda_1 (L_2) f$, an eigenvector corresponding to the smallest eigenvalue of $L_2$. We will extend it to a function on $V$ (as opposed to only $V \setminus \partial G$) by setting it to be 0 in $\partial G$. By checking a single row of the eigenvalue equation, we infer that the function $f$  satisfies, for all $v \in V \setminus \partial G$, 
\begin{align} \label{eq:imp}
\deg(v) f(v) - \sum_{(v,w) \in E} f(w) =  \lambda_1(L_2) f(v) 
\end{align}

We also note that we may assume, without loss of generality, that
$$ \lambda_1(L_2) < \min_{v \in V} \deg(v)$$
since the desired inequality is trivially satisfied otherwise.
 Let us fix an arbitrary $v \in V \setminus \partial G$. Equation \eqref{eq:imp} is equivalent to
$$   \left(1 - \frac{ \lambda_1(L_2)}{\deg(v)} \right) f(v) =  \frac{1}{\deg(v)}\sum_{(v,w) \in E} f(w).$$
This can be turned into an inequality: if $v \in V \setminus \partial G$, then
$$ \frac{1}{\deg(v)}\sum_{(v,w) \in E} f(w) \geq    \left(1 - \frac{ \lambda_1(L_2)}{\min_{v \in V} \deg(v)} \right) f(v).$$
Abbreviating $q =   \lambda_1(L_2)/\min_{v \in V} \deg(v)$ allows us to write that inequality even more concisely: if $v_k \in V \setminus \partial G$ is an arbitrary point that is not in the boundary and if $v_{k+1}$ is a randomly chosen neighbor, then the value of $f$ is typically not much smaller and
$$ \mathbb{E} f(v_{k+1}) \geq (1-q) f(v_k).$$
However, when phrased like that, there is a natural modification of the random walk that ensures that this inequality is always true.  Let us define a random walk on $V$ by choosing
$$ v_{k+1} = \begin{cases} \mbox{a random neighbor of}~v_k \qquad &\mbox{if}~v_k \in V \setminus \partial G \\ 
v_k \qquad &\mbox{if}~v_k \in \partial G. \end{cases}$$
Since $f$ vanishes on $\partial G$, we have the inequality $\mathbb{E} f(v_{k+1}) \geq (1-q) f(v_k)$ independently whether or not $v_k$ is in the boundary or not, it is now true everywhere. Let us now choose $v_0$ to be the vertex in which $f$ assumes its maximal value. We may then deduce the inequality
$$ (1 - q)^k f(v_0) \leq \mathbb{E} f(v_k).$$
Using $\mu_k$ to denote the probability distribution of the random walk started in $v_0$ after $k-$steps, the fact that $f$ vanishes in the boundary and that the initial vertex $v_0$ was chosen so as to maximize $f$, we have
\begin{align*}
 (1 - q)^k f(v_0)  \leq \mathbb{E} f(v_k)  &= \sum_{v \in V} f(v) \mu_k(v) = \sum_{v \in V \setminus \partial G} f(v) \mu_k(v) \\
 &\leq  \sum_{v \in V \setminus \partial G} f(v_0) \mu_k(v) = f(v_0) \cdot \mu_k\left( V \setminus \partial G \right).
\end{align*}
This, however, gives us a lower bound on the probability that a random walk started in $v_0$ will not hit the boundary within the first $k$ steps of its life
\begin{align*}
 \mu_k\left( V \setminus \partial G \right) \geq (1-q)^k.
 \end{align*}
However, using Lemma 1 and letting $k \rightarrow \infty$, we know that this likelihood decays exponentially and
$$ (1-q)^k \leq 4 \cdot (\max_{v \in V} \deg(v)) \cdot 2^{-k/(2\mbox{\tiny diam}(G)^2)}.$$
This implies that $ 2^{1/(2\mbox{\tiny diam}(G)^2)} \leq \frac{1}{1-q}$
from which we deduce
$$   \frac{1}{2}\frac{1}{\mbox{diam}(G)^2} \leq \log_2 \left( \frac{1}{1-q} \right).$$
If $q \geq 1/2$, then
$$  \frac{1}{2}  \frac{1}{\mbox{diam}(G)^2} \leq \frac{1}{2} \leq q \leq \frac{\lambda_1(L_2)}{\min_{v \in V} \deg(v)}$$
and the desired result holds. If $0 < q < 1/2$, then we can use an inequality for real numbers that is true in that range to deduce that
$$  \frac{1}{2}\frac{1}{\mbox{diam}(G)^2} \leq \log_2 \left( \frac{1}{1-q} \right) \leq 2q = 2\frac{\lambda_1(L_2)}{\min_{v \in V} \deg(v)} $$
and the desired inequality follows.
\end{proof}

  The argument is again philosophically related to existing arguments in the continuous setting \cite{georg, rachh} to prove a slightly different result that also has an analogue here: the inequality
  $$  \mu_k\left( V \setminus \partial G \right) \geq \left( 1- \frac{\lambda_1(L_2)}{\min_{v \in V} \deg(v)} \right)^k$$
  shows that a random walk started in the vertex in which $f$ is maximal has the property that a random walk is unlikely to immediately hit the boundary: for something like at least
  $$ k \sim \frac{\min_{v \in V} \deg(v)}{\lambda_1(L_2)} $$
  steps, the likelihood of hitting the boundary is $\leq 1/e \sim 37\%$. This means that there are vertices that are, in this sense, far away from the boundary in the sense of random walks.  In the continuous setting this implies that the domain $\Omega$ has to contain `most' of a ball of a certain size (small pieces of the ball can be missing), this was first observed by Lieb \cite{lieb}. Here, a similar phenomenon holds:  $$ \min_{f:V \rightarrow \mathbb{R} \atop f |_{\partial G} = 0 } \frac{\sum_{(u,v) \in E} (f(u) - f(v))^2}{ \sum_{v \in V} f(v)^2}$$
can only be small if the function $f$ has a large `interior' that is `far away' from the boundary $\partial G$. A version of this was already observed (with $\partial G$ replaced by a neighborhood of the nodal set) in the context of diffusion maps \cite{xiu}.

  \section{Proof of Theorem 3}
 \subsection{Agmon-Allegretto-Piepenbrink.} 
 The proof of Theorem 3 is based on a simple in the style of the ideas of Agmon \cite{agmon}, Allegretto \cite{all1,all2} and Piepenbrink \cite{piep}. We first give a discrete formulation of a statement that is adapted from a continuous version we found in Davies \cite[Theorem 4.2.1]{davies}. After that, we present a minuscule modification that will suffice for our purposes. We may think of the ideas in \S 6.1 as a particular formulation of the classic notion that superharmonic functions give rise to Hardy weights, an idea that has been exploited many times.

 \begin{lemma} Let $G=(V,E)$ be a finite, connected graph and let $W:V \rightarrow \mathbb{R}$ be an arbitrary function.  If there exists a strictly positive function $\phi:V \rightarrow \mathbb{R}_{+}$ such that
 $$ (D-A)\phi +  W \phi \geq 0 \qquad \mbox{entrywise,}$$
then for all $f:V \rightarrow \mathbb{R}$
 $$  \sum_{(u,v) \in E} (f(u) - f(v))^2 + \sum_{v \in V} W(v) f(v)^2 \geq 0.$$
 \end{lemma}

 \begin{proof} The proof is a discrete version of the proof of \cite[Theorem 4.2.1]{davies}. There are some steps in the continuous proof that are structurally different, an application of the product formula for derivatives and a certain local argument that becomes non-local here, however, this is compensated by particularly friendly algebra that leads to nice cancellations.
 We are interested in the positivity of the quadratic form
 $$Q(f) = \sum_{(u,v) \in E} (f(u) - f(v))^2 + \sum_{v \in V} W(v) f(v)^2.$$
 Suppose there exists a strictly positive $\phi > 0$ so that
 $$ (D-A)\phi + W\phi \geq 0 \qquad \mbox{entrywise}.$$
 Then we can define the function $X:V \rightarrow \mathbb{R}$
 $$ X(v) = \frac{1}{\phi(v)} [(A-D)\phi](v).$$
This function satisfies $X(v) \leq W(v).$ We make the ansatz $g = f/\phi$ and compute the quadratic form $Q(f)$. It can be written as
 \begin{align*}
  Q(f) &=  \sum_{(u,v) \in E} (f(u) - f(v))^2 + \sum_{v \in V} W(v) f(v)^2 \\
  &=  \left\langle (D-A) g \phi, g \phi \right\rangle + \sum_{v \in V} W(v) \phi(v)^2 g(v)^2.
 \end{align*}
Emulating an application of the product rule for derivatives, we write
\begin{align*} \left[(D-A) g \phi \right](v)  &= \sum_{(v,w) \in E} ( g(v) \phi(v) - g(w) \phi(w)) \\
&=  \sum_{(v,w) \in E} ( g(v) \phi(v) - g(v) \phi(w) + g(v) \phi(w) -  g(w) \phi(w)) \\
&= g(v) \left[(D-A)\phi\right](v) + \sum_{(v,w) \in E} (g(v) - g(w)) \phi(w),
\end{align*}
This explicit equation now allows us to split the inner product as
\begin{align*}
 \left\langle (D-A) g \phi, g \phi \right\rangle &= \sum_{v \in V} g(v) \phi(v) \left[(D-A) g \phi \right](v) \\
 &=\sum_{v \in V} g(v) \phi(v)g(v) \left[(D-A)\phi\right](v)  \\
 &+  \sum_{v \in V} g(v) \phi(v) \sum_{(v,w) \in E} (g(v) - g(w)) \phi(w)
\end{align*}
The first sum is quite benign and can be written as
$$  \sum_{v \in V} g(v)^2 \phi(v) [(D-A)\phi](v) = - \sum_{v \in V} X(v) f(v)^2.$$
In particular, the first sum can be added to the potential term and 
$$ - \sum_{v \in V} X(v) f(v)^2 +  \sum_{v \in V} W(v) f(v)^2 =  \sum_{v \in V} ( \underbrace{W(v) - X(v)}_{\geq 0}) f(v)^2.$$
The remaining sum $\Sigma_2$ can be written as
\begin{align*}
\sum_{v \in V} g(v) \phi(v) \sum_{(v,w) \in E} (g(v) - g(w)) \phi(w) &= \sum_{v \in V}  \sum_{(v,w) \in E} g(v) \phi(v) (g(v) - g(w)) \phi(w).
\end{align*}
This summation runs over all vertices and then all adjacent edges. In particular, we sum over each edge twice (`once from each side'). This suggests changing the order of summation and instead to sum over the edges. This leads to
\begin{align*}
\Sigma_2 &= \sum_{(v,w) \in E} g(v) \phi(v) (g(v) - g(w)) \phi(w) + g(w) \phi(w) (g(w) - g(v)) \phi(v) \\
&=  \sum_{(v,w) \in E}  \phi(v) \phi(w) \left[ g(v)  (g(v) - g(w)) + g(w) (g(w) - g(v)) \right] \\
&= \sum_{(v,w) \in E} \phi(v) \phi(w) (g(v) - g(w))^2 \geq 0.
\end{align*}
Altogether, we see that $Q(f) \geq 0$ as desired.
\end{proof}

\subsection{A modification.}
Lemma 2 can be thought of as an implementation of the classic approach.  We will need a small modification where we allow the function $\phi$ to not satisfy the inequality in certain places; this does not cause problems as long as the function $f$ vanishes in those places. 

 \begin{lemma} Let $G=(V,E)$ be a finite, connected graph and let $X \subset V$ be an arbitrary subset of vertices. Suppose $W: V \rightarrow V$ is an arbitrary function and that there exists a function $\phi: V \rightarrow \mathbb{R}$ that is strictly positive in $V \setminus X$ such that
 $$ \left[(D-A) \phi \right](v) + W(v) \phi(v) \geq 0 \qquad \mbox{for all}~v \in V \setminus X.$$
Then for all $f:V \rightarrow \mathbb{R}$ that vanish on $X$, meaning $f\big|_{X} = 0$, we have
 $$  \sum_{(u,v) \in E} (f(u) - f(v))^2 + \sum_{v \in V} W(v) f(v)^2 \geq 0.$$
 \end{lemma}
 \begin{proof} The proof proceeds along the same lines although various steps have to be undertaken with greater care. We note the changes while referring to the argument above for things that stay the same. The function $\phi:V \rightarrow \mathbb{R}$ is positive in $V \setminus X$, this means that
  $$ X(v) = \frac{1}{\phi(v)} [(A-D)\phi](v) \qquad \mbox{is defined in}~V \setminus X$$
 and satisfies $X(v) \leq W(v)$ there. If $f:V \rightarrow \mathbb{R}$ is a function vanishing in $X$, then we can define a function $g:V \rightarrow \mathbb{R}$ by setting 
 $$ g(v) = \begin{cases} f(v)/\phi(v) \qquad &\mbox{if}~v \in V \setminus X \\ 0 \qquad &\mbox{if}~v \in X. \end{cases}$$
Since $f$ vanishes in $X$, we have
 $ \sum_{v \in V} W(v) f(v)^2 =  \sum_{v \in V \setminus X} W(v) f(v)^2.$
The quadratic form $Q(f)$ can be written as
 \begin{align*}
  Q(f)  = \left\langle (D-A) g \phi, g \phi \right\rangle + \sum_{v \in V} W(v) \phi(v)^2 g(v)^2.
 \end{align*}
 Just as above, we can expand the inner product and get
 \begin{align*}
 \left\langle (D-A) g \phi, g \phi \right\rangle  &=\sum_{v \in V} g(v) \phi(v)g(v) \left[(D-A)\phi\right](v)  \\
 &+  \sum_{v \in V} g(v) \phi(v) \sum_{(v,w) \in E} (g(v) - g(w)) \phi(w)
\end{align*}
 We start with the first sum.  Since $g \equiv 0$ in $X$, we have
 $$  \sum_{v \in V} g(v)^2 \phi(v) [(D-A)\phi](v) =  \sum_{v \in V \setminus X} g(v)^2 \phi(v) [(D-A)\phi](v).$$
 For $v \in V \setminus X$ we have  
 $ \phi(v) [(D-A)\phi](v) = \phi(v)^2 (-X(v)) \geq  \phi(v)^2 (-W(v)).$
 This leads to the first sum cancelling with the sum containing the potential term, the rest of the argument is as before.
  \end{proof}

\subsection{Proof of Theorem 3}
 \begin{proof} We use Lemma 3. Let $X = \partial G$ and let
Let $\phi:V \rightarrow \mathbb{R}$ be the expected number of steps a random walk needs to hit the boundary $\partial G$. Note that $\phi \equiv 0$ in $\partial G$. We have to compute the Laplacian
$$  \left[(D-A) \phi \right](v) =  \deg(v) \phi(v) - \sum_{(v,w) \in E} \phi(w).$$
However, we can also think about the function $\phi$ dynamically: the expected hitting time in $v$ consists in taking one random step and then the average hitting time of all the neighbors, therefore
$$ \phi(v) = 1 + \frac{1}{\deg(v)} \sum_{(v,w) \in E} \phi(w)$$
 and we deduce $ \left[(D-A) \phi \right](v) = \deg(v)$. This means that with
 $$ W(v) =  - \frac{\deg(v)}{\phi(v)}$$
 we have
 $$  \left[(D-A) \phi \right](v) + W(v) \phi(v) \geq 0 \qquad \mbox{for all}~v \in V \setminus \partial G$$
 and Lemma 3 implies the result.
 \end{proof}

\section{Proof of Theorem 4 and Theorem 5}
\subsection{An abstract ABP estimate.}
We start with an abstract ABP estimate.

\begin{lemma}
Let $G=(V,E)$ be a finite, connected graph, let $X \subset V$ be a subset of the vertices and let $L$ denote the Kirchhoff Graph Laplacian. Then, there exists a constant $C(X) < \infty$ (depending on $G=(V,E)$ and $X$) such that for all $f: V \rightarrow \mathbb{R}$
$$ \max_{v \in V} f(v) \leq  \max_{w \in  X} f(w) + C(X) \| Lf \|_{L^{\infty}(V \setminus X)}.$$
\end{lemma}
\begin{proof}
The Laplacian $L$ annihilates constant functions, we can thus replace 
$$f \rightarrow f - \max_{w \in X} f(w)$$
without changing the inequality. This rescaling ensures that the function satisfies $f \leq 0$ on $X$. The quantity $\| Lf \|_{L^{\infty}(V \setminus X)}$ can only decrease if we replace all negative values of $f$ in $X$ by $0$. We may thus assume that $f \equiv 0$ in $X$ and the inequality that remains to be shown is
 $$ \max_{v \in V} f(v) \leq C(X) \| Lf \|_{L^{\infty}(V \setminus X)}.$$
We may then replace $f \rightarrow \lambda f$ for any $\lambda >0$ without changing the inequality and may thus assume that $\max_{v \in V} f(v) = 1$. If there as no such constant, then we can take a sequence of functions $f_k:V \rightarrow \mathbb{R}$ that vanish on $X$, attain maximal value 1 on $V \setminus X$ and satisfy
$$ 1 \geq k \cdot \| Lf_k \|_{L^{\infty}(V \setminus X)}.$$
Since the graph is finite, the sequence $f_k$ has a convergent subsequence so that $f_{n_k} \rightarrow f^*$ and $\| Lf^* \|_{L^{\infty}(V \setminus X)} = 0$. This shows that in every vertex $v$ in which $f^*$ attains a maximum, all the adjacent vertices also have to attain a maximum and this contradicts the fact that $G=(V,E)$ is connected.
\end{proof}

\begin{center}
\begin{figure}[h!]
\begin{tikzpicture}
\filldraw (0,0) circle (0.06cm);
\filldraw (1,1) circle (0.06cm);
\filldraw (1,-1) circle (0.06cm);
\filldraw (2,1.5) circle (0.06cm);
\filldraw (2,0.5) circle (0.06cm);
\filldraw (2,-1.5) circle (0.06cm);
\filldraw (2,-0.5) circle (0.06cm);
\filldraw (3,1.75) circle (0.06cm);
\filldraw (3,1.25) circle (0.06cm);
\filldraw (3,0.75) circle (0.06cm);
\filldraw (3,0.25) circle (0.06cm);
\filldraw (3,-1.75) circle (0.06cm);
\filldraw (3,-1.25) circle (0.06cm);
\filldraw (3,-0.75) circle (0.06cm);
\filldraw (3,-0.25) circle (0.06cm);
\draw [thick] (0,0) -- (1,1) --(2, 1.5) -- (3, 1.75);
\draw [thick] (0,0) -- (1,-1) --(2, -1.5) -- (3, -1.75);
\draw [thick]   (1,1) --(2, 0.5) -- (3, 0.75);
\draw [thick]   (2, 0.5) -- (3, 0.25);
\draw [thick]   (1,-1) --(2,-0.5) -- (3, -0.75);
\draw [thick]   (2, -0.5) -- (3, -0.25);
\draw [thick]  (2, 1.5) -- (3, 1.25);
\draw [thick]  (2, -1.5) -- (3, -1.25);
\draw [thick] (0,0) circle (0.3cm);
\node at (-0.4, -0.4) {$X$};
\draw [dashed] (3,1.75) -- (4, 1.8);
\draw [dashed] (3,1.75) -- (4, 1.7);
\draw [dashed] (3,1.25) -- (4, 1.3);
\draw [dashed] (3,1.25) -- (4, 1.2);
\draw [dashed] (3,0.75) -- (4, 0.8);
\draw [dashed] (3,0.75) -- (4, 0.7);
 \draw [dashed] (3,0.25) -- (4, 0.3);
\draw [dashed] (3,0.25) -- (4, 0.2);
 \draw [dashed] (3,-1.75) -- (4, -1.8);
\draw [dashed] (3,-1.75) -- (4, -1.7);
\draw [dashed] (3,-1.25) -- (4, -1.3);
\draw [dashed] (3,-1.25) -- (4, -1.2);
\draw [dashed] (3,-0.75) -- (4, -0.8);
\draw [dashed] (3,-0.75) -- (4, -0.7);
 \draw [dashed] (3,-0.25) -- (4,-0.3);
\draw [dashed] (3,-0.25) -- (4, -0.2);
\draw [thick, <->] (0, -2.5) -- (6, -2.5);
\node at (3, -2.8) {$\ell$ \small levels};
\node at (0, -2.2) {\footnotesize $f = 0$};
\node at (1, -2.2) {\footnotesize $f = 1$};
\node at (2, -2.2) {\footnotesize $f = \frac{3}{2}$};
\node at (3, -2.2) {\footnotesize $f = \frac{7}{4}$};
\node at (6, -2.2) {\footnotesize $f = \frac{2^{\ell} - 1}{2^{\ell-1}}$};
\foreach \x in {0,...,40}{
 \filldraw (6,0.09*\x - 1.8) circle (0.06cm);
 \draw [dashed] (6,0.09*\x - 1.8) -- (5,0.09*\x - 1.8);
 \node at (4.5,0) {$\dots$};
}
\end{tikzpicture}
\caption{An example where $C(X)$ is large: $f$, on a binary tree, is harmonic except for the root and the leaves.}
\end{figure}
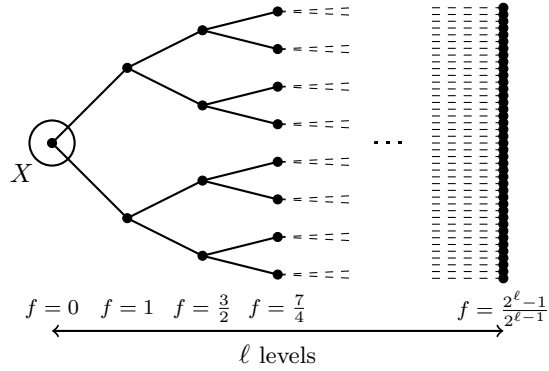
\end{center}

While $C(X) < \infty$ is trivial, $C(X)$ may depend on $G=(V,E)$ and $X$ in a complicated way and it may be exponentially large in various graph parameters. We quickly construct an example for which $C(X)$ is large. The example, shown in Fig. 6, is a binary tree of depth $\ell$. We choose the set $X$ to
be the root of the tree and consider a concrete function $f:V \rightarrow \mathbb{R}$ which vanishes
in the root and only depends on the depth of the vertex. The values converge to $2$ at an exponential rate
in the depth of the level and are chosen so that the function $f$ is actually harmonic except in $X$ and the
very last column (the vertices with degree 1). In the very last column, however, the Laplacian is exponentially small in $\ell$. The arising Alexandrov-Bakelman-Pucci estimate
then shows that
$$ \underbrace{\max_{v \in V} f(v)}_{\sim 2} \leq  \underbrace{\max_{w \in  X} f(w)}_{=0} + C(X) \underbrace{\| Lf \|_{L^{\infty}(V \setminus X)}}_{\sim 2^{-\ell}}.$$
This example implies that the constant $C(X)$ has to satisfy $C(X) \gtrsim 2^{\ell}$. 
 The example is perfectly valid: if $X$ is chosen to be this particular vertex, then one cannot hope for a much better estimate. Our interpretation of the example is the other way around: if, as achieved by Theorem 4, we obtain nice uniform 
 Alexandrov-Bakelman-Pucci estimates, this means that the set $X = \partial G$ has to nicely adapt to the geometry in question to avoid such pathological behavior. Indeed, the boundary $\partial G$ of a tree are simply the leaves. In that case, the root of the tree is part of the interior which, for this particular example, forces us to consider the Laplacian also in the root and we end up with $\| Lf \|_{L^{\infty}(V \setminus \partial G)} \sim 1$ which means that Theorem 4 is trivially satisfied.

\subsection{Proof of Theorem 4} \label{abp:ex}
\begin{proof}
We first rewrite the equation $L f = g$ as $ (D-A)f = g$ or
$$ \deg(v) f(v) - \sum_{(v,w) \in E} f(w) = g(v).$$
This can be rewritten as
$$ f(v) = \frac{g(v)}{\deg(v)} + \frac{1}{\deg(v)} \sum_{(v,w) \in E} f(w)$$
implying, for $v \in V \setminus \partial G$,
\begin{align*}
f(v) &= \frac{g(v)}{\deg(v)} + \frac{1}{\deg(v)} \sum_{(v,w) \in E} f(w)\\
&\leq   \frac{\max\left\{ g(v), 0 \right\}}{\min_{v \in V} \deg(v)} + \frac{1}{\deg(v)} \sum_{(v,w) \in E} f(w) \\
&\leq \frac{\|\max\left\{ g(v), 0 \right\}\|_{L^{\infty}(V\setminus \partial G)}}{\min_{v \in V} \deg(v)} + \frac{1}{\deg(v)} \sum_{(v,w) \in E} f(w)
\end{align*}
If we define a random walk on $V$ in the usual way where $v_{k+1}$ is a randomly chosen neighbor of $v_k \in V \setminus \partial G$, then this inequality ensures that
$$ f(v_k) \leq  \frac{\|\max\left\{ g(v), 0 \right\}\|_{L^{\infty}(V\setminus \partial G)}}{\min_{v \in V} \deg(v)}  + \mathbb{E} f(v_{k+1}).$$
Just as in the proof of Theorem 2, we define $v_{k+1}$ to be $v_k$ whenever $v_k \in \partial G$. In that case, we have the trivial inequality $\mathbb{E} f(v_{k+1}) = f(v_k)$. Starting the random walk in the vertex $v_0 \in V \setminus \partial G$ that maximizes $f$, we deduce that
$$ f(v_0) \leq k \frac{\|\max\left\{ g(v), 0 \right\}\|_{L^{\infty}(V \setminus \partial G) }}{\min_{v \in V} \deg(v)}  + \mathbb{E} f(v_k).$$
Combining Theorem 1 with Markov's inequality shows that after 
$$k= 2\max_{v \in V} \deg(v) \cdot \mbox{diam}(G)^2$$ steps, at least $1/2$ of all the particles are already on the boundary. Choosing that value for $k$, we see that the worst case is when all the random walkers hit points on the boundary that are maximal while the remaining random walkers congregate in the maximal value inside the domain, that gives
$$ \mathbb{E} f(v_k) \leq \frac{1}{2} \max_{v \in \partial G} f(v) + \frac{f(v_0)}{2}.$$
Rearranging gives
$$ f(v_0) \leq \max_{v \in \partial G} f(v) + 2\max_{v \in V}\deg(v)  \mbox{diam}(G)^2 \frac{\|\max\left\{ g(v), 0 \right\}\|_{L^{\infty}(V \setminus \partial G) }}{\min_{v \in V} \deg(v)}$$
which is the desired result.\end{proof}

 \subsection{Proof of Theorem 5}
  \begin{proof} The argument is virtually identical to the proof of Theorem 4, if not even a bit simpler. The equation $Lf = \lambda_2 f$ implies that for every  $v \in V$
 \begin{align*}
 f(v) & = \frac{\lambda_2}{\deg(v)} f(v) + \frac{1}{\deg(v)} \sum_{(v,w) \in E} f(w).
  \end{align*}
  Considering again $v_0 \in V \setminus \partial G$ so that $f$ is maximal and an associated random walk on $V$ where $v_{k+1}$ is a randomly chosen neighbor if $v_k \in V \setminus \partial G$ and $v_{k+1}$ if $v_k$ is in the boundary, we have
  $$ \left(1 - \frac{\lambda_2}{\min_{v \in V} \deg(v)}\right) f(v_{k}) \leq f(v_{k+1})$$
and thus
  $$ \left(1 - \frac{\lambda_2}{\min_{v \in V} \deg(v)}\right)^k f(v_0) \leq \mathbb{E} f(v_k).$$
  Choosing, as above, $k = 2\max_{v \in V} \mbox{diam}(G)^2$ ensures that half the random walks have hit the boundary and
    $$ \left(1 - \frac{\lambda_2}{\min_{v \in V} \deg(v)}\right)^k f(v_0) \leq \frac{1}{2} \max_{w \in \partial G}f(w) + \frac{f(v_0)}{2}$$
 and thus
 $$ \max_{v \in V \setminus \partial G} f(v) \leq   \left(1 - \frac{\lambda_2}{\min_{v \in V} \deg(v)}\right)^{- 2\max_{v \in V} \mbox{\tiny diam}(G)^2}  \max_{w \in \partial G}f(w).$$
 \end{proof}

\section{Proof of Theorem 6}
\begin{proof} Let $\mu$ be a probability measure on the vertices $V$ of the connected graph $G=(V,E)$ that solves
$$ \sum_{v,w \in V} f(d(v,w)) \mu(v) \mu(w) \rightarrow \max.$$
The existence of such a measure is clear from compactness. We will now argue that if $a \in V \setminus \partial G$ is
a vertex that is not in the boundary, then $\mu(a) = 0$.  The proof is by contradiction, suppose $\mu(a) > 0$. We now construct
a new measure $\nu$ by setting
$$ \nu(v) = \begin{cases} \mu(v) \qquad &\mbox{if}~v \neq a ~\mbox{and} ~(a,v) \notin E \\
\mu(v) + \frac{\mu(a)}{\deg(a)} \qquad &\mbox{if}~(a,v) \in E \\
0 \qquad &\mbox{if}~v = a. 
\end{cases}$$
We will show that the measure $\nu$ has a larger energy. Abbreviating
$$ X =  \sum_{v,w \in V} f(d(v,w)) \left( \nu(v) \nu(w) -  \mu(v) \mu(w) \right)$$
we note that nothing happens unless at least one of the vertices is either $a$ or a neighbor of $a$. We use $N(a)$ to denote the neighbors of $a$ and $B(a) = N(a) \cup \left\{a\right\}$ to denote the `ball' containing $a$ and its neighbors. Then, omitting the summand for brevity, we have $\sum_{v,w \in V \setminus B(a)} = 0$, the remaining sums are
\begin{align*}
X &=  \sum_{v \in V \setminus B(a), w \in N(a)} +    \sum_{v \in V \setminus B(a), w = a} \\
&+  \sum_{v \in N(a), w \in V \setminus B(a)} + \sum_{v \in N(a), w \in N(a)} + \sum_{v \in N(a), w = a} \\
&+ \sum_{v =a, w \in V \setminus B(a)} + \sum_{v =a, w \in N(a)} + \sum_{v = a, w = a}.
\end{align*}
We refer to these summands as $\Sigma_1, \dots, \Sigma_8$. Then
\begin{align*}
\Sigma_1 &= \sum_{v \in V \setminus B(a), w \in N(a)} f(d(v,w)) \mu(v) \frac{\mu(a)}{\deg(a)} \\
\Sigma_2 &= \sum_{v \in V \setminus B(a)} f(d(v,a)) (-\mu(v) \mu(a)).
\end{align*}
 We will now argue that the sum of these two terms is positive. First,
\begin{align*}
\Sigma_1 + \Sigma_2 &= \sum_{v \in V \setminus B(a)}  \mu(v) \left(  \left[ \sum_{w \in N(a)} f(d(v,w)) \frac{\mu(a)}{\deg(a)} \right] - f(d(v,a))   \mu(a)  \right) \\
&=  \sum_{v \in V \setminus B(a)}  \mu(v)   \left[ \sum_{w \in N(a)}  \left[f(d(v,w)) - f(d(v,a))\right]  \frac{\mu(a)}{\deg(a)} \right] \\
&=     \frac{\mu(a)}{\deg(a)} \sum_{v \in V \setminus B(a)}  \mu(v)    \sum_{w \in N(a)}  \left[f(d(v,w)) - f(d(v,a))\right]. 
\end{align*}
At this point, we use that $a \notin \partial G$ to argue that the inner sum is positive for each $v \in V$.  Fix $v \in V$. Then
$d(v,a)$ is a certain non-negative integers $d \in \mathbb{N}_{\geq 0}$. We partition the set $N(a)$, the neighbors of $a$ by their distance from $v$. The triangle inequality ensures that each neighbor of $a$ has either distance $d-1, d$ or $d+1$ from $v$. This leads to a partition into $N(a) = N_{d-1} \cup N_d \cup N_{d+1}$. The set of vertices with the same distance does not contribute anything since $d(v,w) = d = d(v,a)$. The remaining two sets contribute
 \begin{align*}
 \sum_{w \in N(a)} f(d(v,w)) - f(d(v,a)) &= (f(d+1) - f(d)) \#N_{d+1} \\
 &+  \left(f(d-1) - f(d)\right) \#N_{d-1}.
 \end{align*}
Since $a \notin \partial G$, we have $\#N_{d+1}  \geq \#N_{d-1}$. Then the expression can be rewritten as
 \begin{align*}
 \sum_{w \in N(a)} f(d(v,w)) - f(d(v,a)) &= (f(d+1) - 2f(d) + f(d-1)) \#N_{d-1} \\
 &+  ( \#N_{d+1} -  \#N_{d-1}) \left(f(d+1) - f(d)\right).
 \end{align*}
We note that $f(d+1) - 2f(d) + f(d-1) > 0$ because of the strict convexity of $f$ and $f(d+1) - f(d) \geq 0$ because $f$ is monotonically increasing. This shows that $\Sigma_1 + \Sigma_2 \geq 0$.  By symmetry, the very same argument with $v,w$ interchanged
 also shows $\Sigma_3 + \Sigma_6 \geq 0$.  The sum $\Sigma_4$ simplifies since
 \begin{align*}
 \Sigma_4 &= \sum_{v,w \in N(a)} f(d(v,w))  \left( \nu(v) \nu(w) -  \mu(v) \mu(w) \right) \\
 &= \sum_{v,w \in N(a)} f(d(v,w))  \left(  \left( \mu(v) + \frac{\mu(a)}{\deg(a)}\right) \left( \mu(w) + \frac{\mu(a)}{\deg(a)} \right) -  \mu(v) \mu(w) \right) \\
 &=  \sum_{v,w \in N(a)} f(d(v,w)) (\mu(v) + \mu(w))  \frac{\mu(a)}{\deg(a)} + \sum_{v,w \in N(a)} f(d(v,w)) \frac{\mu(a)^2}{\deg(a)^2}.
 \end{align*}
Similarly,
  \begin{align*}
 \Sigma_5 &= \sum_{v \in N(a)} f(d(v,a))  \left( \nu(v) \nu(a) -  \mu(v) \mu(a) \right) =\sum_{v \in N(a)} f(d(a,v)) (- \mu(a) \mu(v)) \\
 &= - \sum_{v, w \in N(a)} f(d(a,v))  \frac{\mu(a)}{\deg(a)} \mu(v) = - \sum_{v, w \in N(a)} f(1)  \frac{\mu(a)}{\deg(a)} \mu(v) 
 \end{align*}
 and, by symmetry,  $\Sigma_7 = \Sigma_5$.
Therefore
 \begin{align*}
  \Sigma_4 + \Sigma_5 + \Sigma_7  &=  \sum_{v,w \in N(a)} (f(d(v,w))-f(1)) (\mu(v) + \mu(w))  \frac{\mu(a)}{\deg(a)} \\
  &+ \sum_{v,w \in N(a)} f(d(v,w)) \frac{\mu(a)^2}{\deg(a)^2}. 
 \end{align*}
 The last sum can be used to cancel $\Sigma_8 = - f(0) \mu(a)^2$ since
 $$  \sum_{v,w \in N(a)} f(d(v,w)) \frac{\mu(a)^2}{\deg(a)^2} -   f(0) \mu(a)^2 = \sum_{v,w \in N(a)} (f(d(v,w)) - f(0)) \frac{\mu(a)^2}{\deg(a)^2}$$
 and $f$ is monotonically increasing. It remains to show that
 $$   \sum_{v,w \in N(a)} (f(d(v,w))-f(1)) (\mu(v) + \mu(w))  \geq 0.$$  
 By symmetry, it is enough to show
  $$   0\leq   \sum_{v,w \in N(a)} (f(d(v,w))-f(1)) \mu(v)  = \sum_{v \in N(a)} \mu(v) \sum_{w \in N(a)} (f(d(v,w))-f(1)) $$  
which would be implied if we can show that the inner sum is nonnegative. Here, we use once more that $a \in V \setminus \partial G$.  Fixing a neighbor $v$, this tells us that since $a$ has one neighbor at distance 0 from $v$ (of course $v$ itself), it has to have at least one neighbor at distance 2 and therefore 
$$ \sum_{w \in N(a)} (f(d(v,w))-f(1))  \geq (f(0) - f(1)) + (f(2) - f(1)) \geq 0$$
by convexity of $f$. This concludes the proof.
\end{proof}


\begin{thebibliography}{10}

 \bibitem{agmon} S. Agmon, Lectures on exponential decay of solutions of second order elliptic equations. Bounds
on eigenfunctions of N-body Schr\"odinger operators. Mathematical Notes, Princeton Univ. Press,
Princeton, N.J., 1982


\bibitem{alex} R. Alexander, Generalized sums of distances. Pacific J. Math. 56 (1975), no. 2, 297--304.

 
\bibitem{all1} W. Allegretto, On the equivalence of two types of oscillation for elliptic operators, Pacific J.
Math. 55 (1974), 319-328.

\bibitem{all2} W. Allegretto, Positive solutions and spectral properties of second order elliptic operators, Pacific J.
Math. 92 (1981), 15-25. 



\bibitem{alex3} R. Alexander and K. Stolarsky,
Extremal problems of distance geometry related to energy integrals.
Trans. Amer. Math. Soc. 193 (1974), 1--31.

\bibitem{ale1}  A. D. Aleksandrov, Certain estimates for the Dirichlet problem. Dokl. Akad. Nauk SSSR 134, p. 1001-1004;
translated as Soviet Math. Dokl. 1, p. 1151--1154 (1961)

\bibitem{ale2} A. D. Aleksandrov, Uniqueness conditions and bounds for the solution of the Dirichlet problem, Vestnik
Leningrad. Univ. Ser. Mat. Meh. Astronom. 18, p. 5--29 (1963)
 

\bibitem{bakel}  I. Bakelman, On the theory of quasilinear elliptic equations. Sibirsk. Mat. Z, 179--186 (1961)

\bibitem{hardy1} A. Balinsky, W. Desmond Evans, and Roger T. Lewis. The analysis and geometry of Hardy's inequality. Vol. 1. Cham: Springer, 2015.



\bibitem{bieberbach} L. Bieberbach, \"Uber eine Extremaleigenschaft des Kreises, J.-Ber. Deutsch. Math. Verein. 24 (1915), p. 247--250

\bibitem{bjorck} G. Bjorck, Distributions of positive mass, which maximize a certain generalized energy integral. Arkiv f\"or Matematik, 3 (1956), 255--269.

 \bibitem{brasco} L.  Brasco and B. Ruffini, Compact Sobolev embeddings and torsion functions. In Annales de l'Institut Henri Poincar\'e C, Analyse non lin\'eaire Vol. 34 (2017), p. 817-843.  

\bibitem{burdzy} K. Burdzy and W. Werner, A counterexample to the" hot spots" conjecture, Annals of mathematics (1999): 309--317.


\bibitem{ce} J. Caceres, C. Hernando, M. Mora, I. Pelayo, M. Puertas and C. Seara, On geodetic sets formed by boundary vertices.
Discrete Math. 306 (2006), no. 2, 188--198.

\bibitem{carando} D. Carando, D. Galicer and D. Pinasco, 
Energy integrals and metric embedding theory. 
Int. Math. Res. Not. 2015, no. 16, 7417--7435.

\bibitem{ch} G. Chartrand, D. Erwin, G. Johns and P. Zhang, Boundary vertices in graphs, Discrete Math. 263 (2003), p. 25 -- 34

\bibitem{ch2} G. Chartrand, D. Erwin, G. Johns, and P. Zhang, On boundary vertices in graphs. J. Combin. Math. Combin. Comput., 48:39--53, 2004.


\bibitem{xiu} X. Cheng, M. Rachh and S. Steinerberger,  On the diffusion geometry of graph Laplacians and applications. Applied and Computational Harmonic Analysis, 46 (2019), 674-688.

\bibitem{wxml} N. Chiem, W. Dudarov, C. Lee, S. Lee, K. Liu, A characterization of graphs with at most four boundary vertices, Journal of Combinatorics, to appear.

\bibitem{davies} E. Brian Davies, Heat kernels and spectral theory, Cambridge University Press, 1989

\bibitem{luz} O. Ciaurri and L. Roncal, Hardy’s inequality for the fractional powers of a discrete Laplacian, J. Anal., 26 (2018), p. 211-225.

\bibitem{pont} J. de Dios Pont, Convex sets can have interior hot spots, arXiv: 2412.06344

\bibitem{dep} A. DePavia and S. Steinerberger, Spectral clustering revisited: Information hidden in the Fiedler vector, Foundations of Data 3 (2021), p. 225--249



\bibitem{eroh} L. Eroh and R. Oellermann,
Geodetic and Steiner geodetic sets in 3-Steiner distance hereditary graphs. 
Discrete Math. 308 (2008), no. 18, 4212--4220.
 
\bibitem{fischer} F. Fischer, Hardy inequalities on graphs, Doctoral dissertation, Universit\"at Potsdam, 2024.


\bibitem{fischer} F. Fischer, On the optimality and decay of p-Hardy weights on graphs. Calculus of Variations and Partial Differential Equations, 63 (2024), 162.

  \bibitem{georg} B. Georgiev and M. Mukherjee, On maximizing the fundamental frequency of the complement of an obstacle. Comptes Rendus Mathematique, 356 (2018), 406-411.

  

\bibitem{gupta} S. Gupta, Hardy and Rellich inequality on lattices, Calc. Var. Partial Differential
Equations, 62(2024):Paper No. 81.

\bibitem{gupta2}  S. Gupta, One-dimensional discrete Hardy and Rellich inequalities on integers. J.
Fourier Anal. Appl., 30 (2025): Paper No. 15.


\bibitem{gross} O. Gross, The rendezvous value of metric space. Advances in game theory pp. 49-53 Princeton Univ. Press, Princeton, 1964.

\bibitem{gruber} P. M. Gruber, Convex and discrete geometry.   Springer, 2007.

\bibitem{has} Y. Hasegawa and A. Saito, Graphs with small boundary. Discrete Math. 307:1801--1807, 2007


\bibitem{hinrichs} A. Hinrichs, P. Nickolas and R. Wolf, A note on the metric geometry of the unit ball. Math. Z. 268 (2011), no. 3-4, 887--896.


\bibitem{hoskins} J. G. Hoskins and S. Steinerberger, Towards optimal gradient bounds for the torsion function in the plane. The Journal of Geometric Analysis, 31 (2021), 7812-7841.



\bibitem{keller} M. Keller, Y. Pinchover, F. Pogorzelski, Optimal Hardy inequalities for Schr\"odinger operators on graphs, Communications in Mathematical Physics, 358 (2018), 767--790.

\bibitem{keller2} M. Keller, Y. Pinchover, F. Pogorzelski, From Hardy to Rellich inequalities on graphs. Proceedings of the London Mathematical Society, 122 (2021), 458-477.

\bibitem{keller3} M. Keller and M. Nietschmann, Optimal Hardy inequality for fractional Laplacians
on the integers. Ann. Henri Poincar\'e, 24 (2023), p. 2729--2741.

\bibitem{keller4} M. Keller, Y. Pinchover, and F. Pogorzelski. Criticality theory for Schr\"odinger
operators on graphs. J. Spectr. Theory, 10(2020): p. 73--114, 2020.

\bibitem{ken1}  J. Kennedy and J. Rohleder, On the hot spots of quantum trees. PAMM, 18 (2018), e201800122.

\bibitem{ken2}  J. Kennedy and J. Rohleder, On the hot spots of quantum graphs, Comm. Pure Appl. Anal. 20 (2021), 3029-3063


\bibitem{led} R. Lederman and S. Steinerberger, Extreme values of the Fiedler vector on trees, Linear Algebra and its Applications 703 (2024), p. 528-555

\bibitem{lieb} E. H. Lieb, On the lowest eigenvalue of the Laplacian for the intersection of two domains. Inventiones mathematicae, 74(3), 441--448.

\bibitem{lu} J. Lu and S. Steinerberger, A dimension-free Hermite-Hadamard inequality via gradient estimates for the torsion function. Proceedings of the American Mathematical Society, 148 (2020), 673-679.



\bibitem{mariano} P. Mariano, H. Panzo and J. Wang, Improved upper bounds for the Hot Spots constant of Lipschitz domains. Potential Analysis, 59 (2023), 771-787.

\bibitem{mu} T. M\"uller, A. P\'or, and J.-S. Sereni, Lower bounding the boundary of a graph in terms of its maximum or minimum degree. Discrete Math. 308, p. 6581--6583, 2008.


\bibitem{pel} I. Pelayo, Geodesic convexity in graphs. Vol. 577. New York: Springer, 2013.

\bibitem{hardy2} L.-E. Persson, A. Kufner, N. Samko, Weighted inequalities of Hardy type. World Scientific Publishing Company. 2017.


\bibitem{piep} J. Piepenbrink, Nonoscillatory elliptic equations, J. Differential Equations 15 (1974), 541-550. 
 

\bibitem{polya} G. P\'olya, Torsional rigidity, principal frequency, electrostatic capacity and symmetrization. Quart. Appl.
Math. 6, 267-277 (1948)

\bibitem{puc1} C. Pucci, Limitazioni per soluzioni di equazioni ellittiche. Ann. Mat. Pura Appl. 74, 1966, p. 15--30.
\bibitem{puc2}  C. Pucci, Operatori ellittici estremanti, Ann. Mat. Pura. Appl., (4), 72, (1966), p. 141--170.


  \bibitem{rachh} M. Rachh and S. Steinerberger, On the location of maxima of solutions of Schr\"odinger's equation. Communications on Pure and Applied Mathematics, 71 (2018), 1109-1122.

\bibitem{ray} Rayleigh, The theory of sound, MacMillan, New York, 1877,1894; Dover, New York, 194

\bibitem{schon1} I. J. Sch\"onberg, Remarks to Maurice Fr\'echet's article "Sur la d\'efinition axiomatique
d'une classe d'espaces vectoriels distanci\'es applicables vectoriellement sur l'espace de Hilbert, Annals of Math. 36 (1935), 724-732. 

\bibitem{schon2} I. J. Sch\"onberg, On certain metric spaces arising from Euclidean spaces by a change of metric and
their imbedding in Hilbert space. Annals of Math. 38 (1937), 787-793. 

\bibitem{stein0} S. Steinerberger, An endpoint Alexandrov Bakelman Pucci estimate in the plane. Canadian Mathematical Bulletin, 62 (2019), 643-651.

\bibitem{stein} S. Steinerberger, The Boundary of a Graph and its Isoperimetric Inequality, Discrete Applied Mathematics 338 (2023), p. 125-134

\bibitem{stein2}  S. Steinerberger,  Sums of distances on graphs and embeddings into Euclidean space. Mathematika, 69 (2023), 600-621.

\bibitem{stein3} S. Steinerberger, The first eigenvector of a distance matrix is nearly constant,  Discrete Mathematics, 346 (2023), 113291.

\bibitem{stein4} S. Steinerberger,  Curvature on graphs via equilibrium measures, Journal of Graph Theory 103.3 (2023): 415-436.

\bibitem{stein5} S. Steinerberger, An upper bound on the hot spots constant, Rev. Mat. Iberoam. 39 (2023), no. 4, 1373--1386.


\bibitem{stein6} S. Steinerberger, The Hermite-Hadamard inequality in higher dimensions. The Journal of Geometric Analysis, 30 (2020), 466-483.

\bibitem{talenti} G. Talenti, Elliptic equations and rearrangements. Ann. Scuola Norm. Sup. Pisa Cl. Sci. (4) 3, 697--718
(1976)

\bibitem{thomassen} C. Thomassen, The rendezvous number of a symmetric matrix and a compact connected metric space. Amer. Math. Monthly 107 (2000), no. 2, 163--166.

\bibitem{wald1} A. Wald, On cumulative sums of random variables, The Annals of Mathematical Statistics. 15 (1944): 283--296. 

\bibitem{wald2} A. Wald, Some generalizations of the theory of cumulative sums of random variables, The Annals of Mathematical Statistics. 16 (1945): 287-293.

\bibitem{wolf2} R. Wolf, On the average distance property and certain energy integrals.
Ark. Mat. 35 (1997), no. 2, 387-400.

\end{thebibliography}
\end{document}